\documentclass[11pt,letterpaper]{article}

\usepackage[final]{changes}
\colorlet{Changes@Color}{red}

\usepackage[english]{babel}
\usepackage{amsmath}
\usepackage{amsthm}
\usepackage{multirow}
\usepackage{graphicx}
\usepackage{fullpage}
\usepackage{amsfonts}
\usepackage{hyperref}
\usepackage{url}
\usepackage[affil-it]{authblk}
\usepackage{float}
\usepackage[para]{threeparttable}
\usepackage{natbib}
\usepackage{booktabs}
\usepackage[onehalfspacing]{setspace}
\usepackage{enumerate}
\usepackage{mathtools}
\usepackage{relsize}
\hypersetup{colorlinks=true, linkcolor=blue, citecolor=blue, filecolor=blue, urlcolor=blue}

\usepackage[normalem]{ulem}

\usepackage{comment}

\newtheorem{thm}{Theorem}

\usepackage{pifont}
\usepackage{lscape}
\usepackage{algorithm,algpseudocode}
\algnewcommand{\Initialize}[1]{%
  \State \textbf{initialization:}
  \Statex \hspace*{\algorithmicindent}\parbox[t]{.8\linewidth}{\raggedright #1}
}

\allowdisplaybreaks[4]

\newcommand{\p}[1]{\mathbb{P} \left\{ #1 \right\}}
\newcommand{\e}[1]{\mathbb{E} \left[ #1 \right]}

\newcommand{\var}[1]{\mathrm{var} \left( #1 \right)}
\newcommand{\cov}[1]{\mathrm{cov} \left( #1 \right)}

\newcommand{\xh}{x}

\makeatletter 
\let\orgdescriptionlabel\descriptionlabel 
\renewcommand*{\descriptionlabel}[1]{%
\let\orglabel\label   
\let\label\@gobble   
\phantomsection   
\edef\@currentlabel{#1}%
\let\label\orglabel   
\orgdescriptionlabel{#1}%
} 
\makeatother 
\renewenvironment{description}[1][0pt]
  {\list{}{\labelwidth=0pt \itemsep=0pt \leftmargin=#1
   }}
  {\endlist}
\begin{document}

\title{Variance Reduction for Sequential Sampling\\ in Stochastic Programming}
\author[1]{Jangho Park\thanks{\texttt{park.1814@osu.edu}}}
\affil[1]{{\small Department of Integrated Systems Engineering, the Ohio State University, Columbus, Ohio 43210}}
\author[2]{Rebecca Stockbridge\thanks{\texttt{rebecca.stockbridge@gmail.com}}}
\affil[2]{{\small Birmingham, Michigan}}
\author[1]{G\"{u}zin Bayraksan\thanks{\texttt{bayraksan.1@osu.edu}; Corresponding author}}
\date{}
\maketitle
\vspace*{-0.3in}
\begin{abstract}   
\noindent
This paper investigates the variance reduction techniques Antithetic Variates (AV) and Latin Hypercube Sampling (LHS) when used for sequential sampling in stochastic programming and presents a comparative computational study. 
It shows conditions under which the sequential sampling with AV and LHS satisfy finite stopping guarantees and are asymptotically valid, discussing LHS in detail.
It computationally compares their use in both the sequential and non-sequential settings through a collection of two-stage stochastic linear programs with different characteristics. 
The numerical results show that while both AV and LHS can be preferable to random sampling in either setting, LHS typically dominates in the non-sequential setting while performing well sequentially and AV gains some advantages in the sequential setting.  
These results imply that, 
given the ease of implementation of these variance reduction techniques, armed with the same theoretical properties and improved empirical performance relative to random sampling, AV and LHS sequential procedures present attractive alternatives in practice for a class of stochastic programs.\\

\vspace*{-0.03in}


 \end{abstract}

\section{Introduction}
\label{sec:intro}
\vspace*{-0.027in}

We consider stochastic optimization problems of the form
\begin{align}
\tag*{$(\text{SP})$}
\min_{x \in X}\e{f(x, \tilde{\xi})}= \min_{x \in X}\int_{\Xi}f(x,\xi)P(d\xi), \label{SP}
\end{align}
where $X \subseteq \mathbb{R}^{d_x}$ represents the set of constraints the decision vector $x$ must satisfy. 
The random vector $\tilde{\xi}$ on ($\Xi, \mathcal{B}, P$) has support $\Xi \subseteq \mathbb{R}^{d_{\xi}}$, where $\mathcal{B}$ is the Borel $\sigma$-algebra on $\Xi$. 
The distribution of $\tilde{\xi}$, denoted $P$, is assumed not to depend on $x$, and the expectation operator $\mathbb{E}$ is taken with respect to $P$. 
The function $f :X \times \Xi \rightarrow \mathbb{R}$ is assumed to be a Borel measurable, real-valued function, with inputs being the decision vector $x$ and a realization $\xi$ of the random vector $\tilde{\xi}$. 
We will further narrow down the problem class in  Section \ref{sec:vr}.

A major difficulty with \ref{SP} is that it is typically very difficult---and often impossible---to solve exactly.  
This is because for practical problems, $\e{f(x, \tilde{\xi})}$  is a multi-dimensional integral of a complicated function $f$.  
Therefore, it is often impossible to calculate $\e{f(x, \tilde{\xi})}$ exactly for a fixed $x$, let alone optimize it over $x \in X$. 
In these cases, Monte Carlo sampling provides an attractive approximation method.

Consider the Sample Average Approximation (SAA) of \ref{SP} formed using a random sample of size $n$, $\{\tilde{\xi}^1, \tilde{\xi}^2,\ldots,\tilde{\xi}^{n}\}$:\vspace*{-0.15in}
\begin{align}
\tag*{$(\text{SP}_{n})$}
\min_{x \in X} \frac{1}{n}\sum_{i=1}^{n} f(x,\tilde{\xi}^i).  \label{SP_nk}
\end{align}
Even though the asymptotic properties of the optimal value and the set of optimal solutions of SAA have been studied \citep[e.g.,][]{shaDR:09}, in practice, it is important to find high-quality approximate solutions to \ref{SP} using modest sample sizes.

Sequential sampling provides one way to achieve this goal.  
It sequentially increases the sample size, assessing the quality of approximate solutions, until an arguably high-quality solution is identified. 
This paper aims to improve the reliability of existing sequential sampling procedures with minimal computational effort through easy-to-implement variance reduction techniques.
Before discussing the contributions of this paper, let us discuss why sequential sampling may be preferred.

Consider instead solving a single SAA with a fixed (possibly large) sample size.   
In this case, one could try to identify a sample size 
that probabilistically guarantees  the  solutions  are of high quality. Such sample sizes for SAA have been studied, e.g., by \cite{shahom:00,Kleywegt2001,shaHK:02}.
These sample-size estimates provide valuable insight into the use of Monte Carlo sampling techniques. However, they are typically not suited for practical use: They contain quantities that are difficult to estimate, and even if these quantities are known or estimated, the sample sizes can be often overly conservative.  

Alternatively, one could solve a single SAA problem of a desired sample size $n$ and then assess the quality of the approximate solution found in this way (and stop).  Many methods have been developed for quality assessment via Monte Carlo sampling by estimating optimality gaps \citep{mulvey1995new,mak1999monte}, testing optimality conditions \citep{higle_sen_96b,shapiro1998simulation,royset_12,Sakalauskas2002558}, and checking the stability of solutions \citep{kaut_wallace_07}, among others.   
If this assessment reveals a high-quality solution, then this method works fine. 

However, it may happen that this quality assessment reveals a low-quality solution.  
This could be, e.g., the result of sampling error. Alternatively, because the sample size $n$ was fixed and there is no {\it a priori} knowledge of an adequate sample size in getting a good solution, this procedure may indeed yield a low-quality solution.
Then, a user is left with how to increase the sample size in order to obtain a high-quality solution. 

This is where the sequential sampling procedures help practitioners.  
By providing rules to increase the sample size and to stop sampling, they find a sample size that is not typically overly conservative to obtain high-quality solutions. 
We note that the use of variance reduction techniques in this setting is paramount because they can significantly improve the reliability of the estimators, helping to stop earlier, increasing the probability of correctly identifying high-quality solutions, and providing substantially smaller confidence interval widths on solution quality.


Motivated by the potential benefits of variance reduction techniques within sequential sampling, this paper studies the use of sampling techniques Antithetic Variates (AV) and Latin Hypercube Sampling (LHS) within sequential sampling for \ref{SP}. 
A second motivation of the paper is to provide a comparative computational study to gain insights into when AV and LHS should be used. Toward this end, the paper compares their performances relative to random sampling in both the sequential and non-sequential settings.  
\vspace*{-0.13in}

\subsection{Related Literature}
\label{ssec:lit}
\vspace*{-0.03in}

Numerous variance reduction techniques have been studied for stochastic programming. 
We refrain from a thorough review on this topic and discuss some commonly-investigated alternative sampling techniques.
We direct the readers to \cite{hombay_14,hombay_15} for further discussion and references on this topic.
This work differs from most of this literature in that it considers variance reduction in a sequential sampling setting, which has received little attention. 

The commonly-investigated alternative sampling techniques for variance reduction in \ref{SP} include 
(i) importance sampling \citep[e.g.,][]{Barrera2016,dantzig1990parallel,higle1998variance,infanger1992monte,GlynnI2013,kozmik_morton_15}, (ii) (randomized) Quasi-Monte Carlo (QMC) \citep[e.g.,][]{pennanen2005epi,koivu2005variance,homem2008rates,homem2011sampling,heitsch_etal_16}, (iii) AV \citep[e.g.,][]{freLT:10,higle1998variance,koivu2005variance}, and (iv) LHS \citep[e.g.,][]{bailey1999response,freLT:10,drew2007quasi,deMatos2016,linderoth_etal_06}. 
Importance sampling requires finding a sampling distribution, which may be difficult to obtain for some \ref{SP}. 
Randomized QMC has been found effective for a class of stochastic programs, but generating QMC observations can become costly and the method can lose effectiveness as the dimension of $\tilde{\xi}$ increases for some problems. 
Dimension-reduction strategies for QMC can themselves be computationally burdensome \citep{drew2006quasi}.  

In contrast, both AV and LHS are easy to implement and work relatively well as the dimension of $\tilde{\xi}$ increases. 
It is important to note that in the context of \ref{SP}, many variance reduction techniques---including LHS and to a lesser extent AV---have been found to also reduce bias, further improving the quality of sampling-based estimators; see, e.g., the computations in \cite{freLT:10} and \cite{stockbridge_bayraksan_2015}. 
Thus, overall, AV and LHS constitute good choices for variance and bias reduction with virtually no added computational cost.

We now discuss how this paper relates to and differs from a series of papers on solution quality assessment. 
Beginning with \cite{mulvey1995new} and \cite{mak1999monte}, several researchers studied how to assess solution quality for a given candidate solution via estimating optimality gaps and optimal values \citep{bayraksan2006assessing,drew2007quasi,love_bayraksan_tomacs_15,chen2014validating,stockbridge_bayraksan_13,freLT:10}. 
In particular, \cite{stockbridge_bayraksan_2015} and \cite{bayraksan_18} study AV and LHS for optimality gap estimation in a non-sequential setting.
Here, we use them within sequential stopping rules.
 
This change of focus---from non-sequential to sequential---causes important differences. First, in the non-sequential setting, both the candidate solution and the sample size are \textit{fixed}.  In this paper, however, they change from iteration to iteration.  So now, {\it both} the candidate solution and the sample size when the sequential procedures stop become {\it random variables}.  This requires new probabilistic tools for analysis and thus, different proofs.  
In addition to changes in the theoretical analysis, sequential methods also differ from non-sequential methods computationally. 
This is because they may induce bias due to the stopping rule, see, e.g., \cite{glynn1992asymptotic}, and hence require more care.

The sequential use of optimality gap estimates has been investigated in papers \citep{morton1998stopping,bayraksan2011sequential,pierre2011combined,bayraksan2012fixed}, which provide alternative stopping rules under different assumptions.
They have been used to solve, e.g., stochastic scheduling problems
\citep{keller_bayraksan_10,wu_sioshansi_16}.
\vspace*{-0.13in}

\subsection{{Contributions}}
\label{ssec: contributions}
\vspace*{-0.03in}

This paper rigorously investigates, to the best of our knowledge for the first time, the use of AV and LHS within the sequential sampling framework of \citet{bayraksan2011sequential}  both theoretically and computationally.  
It provides conditions under which sequential sampling with AV and LHS stop in a finite number of iterations and yield asymptotically valid confidence intervals on the quality of solutions. 
While both AV and LHS have been investigated earlier for \ref{SP}, this paper differs from the existing literature by focusing on the sequential aspects.

The paper  builds upon the sequential framework of \cite{bayraksan2011sequential} in several ways.
First and foremost, by using a different set of conditions and making important modifications to the proofs for LHS, the paper rigorously justifies the use of LHS sequential procedures. 
Theoretical results for AV generalize with only slight adjustments. That said, we provide alternative proofs for the AV sequential procedures based on the proofs presented in this paper.

The paper considers three alternate sets of assumptions for LHS sequential procedures; these will be precisely defined in Section \ref{sec:the}. 
The first relates the upper bounds on large deviation probabilities between LHS and random sampling \citep{drew2012some}. 
This assumption is satisfied, e.g., when $f(x,\cdot)$ is additive or component-wise monotone. 
If the first assumption cannot be verified, an additional condition on the sequential procedure's parameters can be used. 
This additional condition is based on the residual term in a functional ANOVA decomposition of $f$. 
The third assumption, which is also present and needed in the first two cases to invoke the Central Limit Theorem (CLT) for LHS, just considers boundedness of $f$. 
The first two assumptions, while being more restrictive, require only a sublinear growth of the sample sizes. In contrast, without either of the first two assumptions, the less restrictive third assumption demands a superlinear growth.

Both sublinear and superlinear growths of the sample sizes have been investigated in \cite{bayraksan2011sequential};
however, 
the assumptions and proofs for LHS in this paper are entirely new. To the best of our knowledge, the first two assumptions of this paper have never been used in the context of sequential sampling for stochastic programming. 
A variant of the third assumption, namely, finiteness of the second moment of $f$, has been considered in the earlier work. 
In fact, we use this variant of the third assumption for the AV sequential procedures with superlinear growth.  However, even in this third case, the proofs for LHS sequential procedures require  modifications.

Another main contribution of the paper is a comparative computational study.
It not only compares the use of AV and LHS to random sampling in the sequential setting, but also contrasts their use in the non-sequential setting.
We are not aware of any works that examine  the sequential versus non-sequential uses.  
This analysis yields that AV can gain some advantages in the sequential setting (unlike the non-sequential setting) and LHS typically dominates in the non-sequential setting. 
This is because AV tends to generate smaller sample variances in this context. 
A similar behavior of the AV sample variance estimator in the non-sequential setting was observed earlier \citep{stockbridge_bayraksan_2015}. 
However, no explanation was given. 
This paper provides an asymptotic explanation for this behavior for the first time, and it further investigates its implications in both the sequential and non-sequential settings.

We end with a remark that the proposed sequential sampling procedures with AV and LHS can either be used as stand-alone procedures, solving SAAs with increasing sample sizes, or be embedded in other sampling-based algorithms.  
Internal stopping rules have been investigated for various methods 
\citep[e.g.,][]{higle1991statistical, higle1996stochastic,norkin1998branch,lan_nemirovski_shapiro_12,dumskis2016nonlinear,sen2014mitigating}. 
An explicit assumption in our framework is that the samples that assess solution quality are {\it independent} of the samples that generate solutions.
This independence assumption ensures a valid solution quality assessment. 
Consequently, the sequential procedures with AV and LHS could complement these internal quality checks to validate high-quality solutions for various sampling-based methods.
\vspace*{-0.09in}

 \subsection{Organization}
\vspace*{-0.02in}

The rest of the paper is organized as follows. 
The next section reviews the solution quality assessment procedures and the sequential framework used in this paper.   
Section \ref{sec:vr} applies AV and LHS to the sequential framework, and 
Section \ref{sec:the} provides conditions under which the proposed procedures are asymptotically valid and stop at finite sample sizes. 
We compare the performance of the proposed sequential sampling procedures numerically in Section \ref{sec:exp}. 
Finally, Section \ref{sec: con} ends the paper with conclusions and future work. 
Throughout the rest of the paper, we refer to  independent and identically distributed sampling simply as IID. 
\vspace*{-0.2in}

\section{Background on Sequential Sampling with IID}\vspace*{-0.05in}
\label{sec:background}

We first review two procedures for solution quality assessment, the Single Replication Procedure (SRP) and the Averaged Two-Replication Procedure (A2RP) \citep{bayraksan2006assessing} in Section \ref{ssec:solqual}, which are subsequently used within the sequential framework reviewed in Section \ref{ssec:overview}. 
We will soon apply AV and LHS to this framework.\vspace*{-0.05in}

\subsection{Solution Quality Assessment}\vspace*{-0.05in}
\label{ssec:solqual}

Let us begin with some notation:  $z^*$ denotes the optimal value and  $X^*=\text{argmin}_{x \in X} \e{f(x, \tilde{\xi})}$ denotes the set of optimal solutions of \ref{SP}. 
An optimal solution is represented by $x^* \in X^*$. 
Similarly, $z^*_{n}$ denotes the optimal value and $x^*_{n}$ denotes an optimal solution of \ref{SP_nk}.  
Because it is important to distinguish between sampling schemes, we will shortly add a notation to $z^*_{n}$  and $x^*_{n}$ to identify the sampling method. 
We also distinguish between a {\it candidate solution} $\xh \in X$, which is a feasible solution that may or may not be optimal, and an {\it optimal solution} $x^*$.

We assess the quality of a given candidate solution $\xh$ by its optimality gap---the lower the gap, the higher the quality. 
The optimality gap of $\xh$, denoted $\mathcal{G}_{\xh}=\e{f(\xh,\tilde{\xi})}-z^* \ge 0$, is often impossible to calculate.  So, SRP estimates this quantity, along with its sample variance, via
\begin{subequations}
\label{eq:srp}
\begin{equation}
  \texttt{GAP}^{S}_{\mathcal{I},n}(\xh)=
\frac{1}{n}\sum_{i=1}^{n} f(\xh,\tilde{\xi}^i)-z_{\mathcal{I},n}^* = \frac{1}{n}\sum_{i=1}^{n} \left( f(\xh,\tilde{\xi}^i)-f(x_{\mathcal{I},n}^*,\tilde{\xi}^i) \right), \label{eq:gapsrp}\vspace*{-0.05in}
\end{equation}    
\begin{equation}
  \texttt{SV}^{S}_{\mathcal{I},n}(\xh) = 
\frac{1}{n-1} \sum^{n}_{i=1}\left[\left(f(\xh,\tilde{\xi}^i)-f(x^*_{\mathcal{I}, n},\tilde{\xi}^i)\right)-\texttt{GAP}^{S}_{\mathcal{I},n}(\xh)\right]^2. \label{eq:svsrp}
\end{equation}
\end{subequations}
In the above estimators, the superscript $S$ denotes SRP and the subscript $\mathcal{I}$ denotes IID. 
In particular, the subscript $\mathcal{I}$ is added to $z_{\mathcal{I},n}^*$ and $x_{\mathcal{I},n}^*$. 

For small sample sizes, SRP has been found to underestimate the optimality gap. 
To improve the small-sample behavior of SRP, A2RP averages two independent SRP estimators of size $n/2$ and pools them to obtain the optimality gap estimators: 
\begin{equation}
\label{eq:a2rp}
\texttt{GAP}^{A}_{\mathcal{I},n}(\xh)=\frac{1}{2}\left(\texttt{GAP}^{S,1}_{\mathcal{I},n/2}(\xh)+\texttt{GAP}^{S,2}_{\mathcal{I},n/2}(\xh)\right)
\ \ \ \text{   and   }\ \ \
\texttt{SV}^{A}_{\mathcal{I},n}(\xh)=\frac{1}{2}\left(\texttt{SV}^{S,1}_{\mathcal{I},n/2}(\xh)+\texttt{SV}^{S,2}_{\mathcal{I},n/2}(\xh)\right).
\end{equation}
The A2RP estimators now have superscript $A$. 
Both procedures use a total sample size of $n$.

The SRP---and hence the A2RP---optimality gap estimator is biased: $\e{\texttt{GAP}^{S}_{\mathcal{I},n}(\xh)}=\e{f(\xh,\xi)}-\e{z^*_{n}} \geq \mathcal{G}_{\xh}$.
This is because $z^*_{n}$ is a biased estimator of $z^*$, $\e{z^*_{n}} \leq z^*$.
	The bias shrinks as the sample size increases \citep{mak1999monte}.
However, compared to IID, 
 LHS and AV have been observed to reduce this bias without increasing the sample size \citep{freLT:10}. 
\vspace*{-0.05in}

\subsection{Sequential Sampling Procedure}\vspace*{-0.05in}
\label{ssec:overview}

Suppose we have at hand a method that generates a sequence of candidate solutions, e.g., by solving SAAs. 
Sequential sampling briefly works as follows:   
It assesses the quality of a candidate solution with an initial sample size by estimating its optimality gap. 
If this optimality gap estimate is small enough according to a stopping criterion, the procedure stops.  
Otherwise, it takes another candidate solution, assesses its quality with an increased sample size, and repeats this procedure until the optimality gap estimate falls below a certain level.

To formally describe the procedure, we use subscript $k$ to denote the iteration number. So, $\xh_k$ and $n_k$ denote the candidate solution and the sample size at iteration $k$, respectively.  
The quantities $\texttt{GAP}_{n_k}(\xh_k)$ and $\texttt{SV}_{n_k}(\xh_k)$ represent a {\it generic} optimality gap point estimate and its sample variance at iteration $k$. These could be the SRP estimators in \eqref{eq:srp} or, assuming the sample size $n_k$ is divisible by two, the A2RP estimators in \eqref{eq:a2rp}.
The procedure takes as input the parameters $h>h^{\prime}>0$,  $\varepsilon>\varepsilon^{\prime}>0$, and $\alpha \in (0,1)$ for the desired confidence level $1-\alpha$ on the quality of solutions.  
Here, the parameters $\varepsilon,\varepsilon^{\prime}$ are small numbers, e.g., $\varepsilon=2\times 10^{-7}, \varepsilon^{\prime}=10^{-7}$.  
They are needed for theoretical results, and in practice, they serve to correct any numerical instabilities.

The procedure terminates at the first iteration when the following stopping rule is satisfied:\vspace*{-0.03in}
\begin{align}
T=\inf_{k \geq 1}  \left\{ k:\texttt{GAP}_{n_k}(\xh_k)\leq h' \sqrt{\texttt{SV}_{n_k}(\xh_k)} + \varepsilon' \right\}.  \label{T}
\end{align}
The above rule identifies a solution with a small optimality gap as a fraction of its standard deviation.  
So, if the variability is high, one is willing to stop with a larger optimality gap estimate, and if the variability is low, one stops with a smaller optimality gap estimate. 
This \textit{relative} stopping rule could be desirable in cases when little is known about the variability of the problem, for instance, when small sample sizes are used for assessing solution quality. Otherwise, e.g., when using large sample sizes for quality assessment with a high computational budget, it is possible to adjust the stopping rule \eqref{T} by including an additional condition $\left\{\sqrt{\texttt{SV}_{n_k}(\xh_k)}\leq b\right\}$, where $b$ is a positive real constant without affecting desired theoretical properties \citep{bayraksan2011sequential}. In this case, the procedure stops when the optimality gap falls below an {\it absolute} (not relative) value.

When the stopping rule (\ref{T}) is satisfied, the procedure returns solution $\xh_T$ and an approximate $(1-\alpha)$-level confidence interval\vspace*{-0.03in}
\begin{equation*} 
\texttt{CI}_{n_T}(\xh_T)=\left[0,h\sqrt{\texttt{SV}_{n_T}(\xh_T)}+\varepsilon\right]
\end{equation*} 
on its optimality gap, $\mathcal{G}_{\xh_T}$.
{If the additional stopping criterion that also controls $\sqrt{\texttt{SV}_{n_k}(\xh_k)}$ is used, the confidence interval is updated to $\left[0,hb+\varepsilon\right]$. Throughout the rest of the paper, we focus on the relative stopping rule \eqref{T} without the additional condition.}
Algorithm \ref{alg:seq} summarizes the sequential sampling procedure. 

When using IID, under a finite moment generating assumption, if the sample sizes are increased according to\vspace*{-0.15in} 
\begin{align}
n_k \geq \left(\frac{1}{h-h'}\right)^2(c_p+2p\ln^2k),  \label{nk_IID}
\end{align}
where $c_p=\max\left\lbrace2\ln\left(\sum^{\infty}_{j=1}j^{-p\ln j}/\sqrt{2\pi}\alpha\right),1\right\rbrace$ and $p>0$,
the sequential sampling stops in a finite number of iterations with probability one (w.p.1) and
asymptotically \added{as $h \downarrow h'$ (i.e., as the minimal sample-size requirements given in \eqref{nk_IID} grow)} produces a $(1-\alpha)$-level confidence interval on $\mathcal{G}_{\xh_T}$. 
Note that the sample-size increase formula in \eqref{nk_IID} is sublinear in the iteration number $k$. 
Even though it is sublinear in $k$, it is proportional to $(h-h')^{-2}$, which can be potentially large when $h-h'$ is small. In practice, an initial sample size $n_1$ is given as input to the sequential procedure. This initial sample size is used to determine $h-h'$ through $\left(h-h'\right)^2 \geq \frac{c_p}{n_1}$ by setting $k=1$ in \eqref{nk_IID}. In our computational experiments, due to modest initial sample sizes \added{$n_1$ around 100--500},  
we also obtain modest $h-h'$ values around 0.2--0.4  (Table \ref{tb: para} in Section \ref{ssec: es}).
Under less restrictive assumptions than a finite moment generating function, the sample sizes must be increased at a superlinear rate. 
We will soon recover similar theoretical results---with some modifications---under certain conditions when using AV and LHS.
The required updates and adjustments to accommodate these alternative sampling techniques will be discussed in Sections \ref{sec:vr} and \ref{sec:the}.

\begin{algorithm}[tb]
  \caption{Sequential sampling framework}
  \label{alg:seq}
{\footnotesize 
  \begin{algorithmic}[1]
    \Initialize{Set $k=0$. Select $\alpha\in (0,1)$, and  $h>h^{\prime}>0$, $\varepsilon > \varepsilon^{\prime}>0$. \label{Initialize}
    STOP = No.\smallskip 
}
    \Procedure{Sequential sampling}{}
    \While{STOP = No}
	    \State{$k=k+1$}
	    \State{\textsc{Generate Candidate Solution:}\ \  $\xh_k$} \label{step1}
	    \State{\textsc{Sample:}\ \  Decide sample size $n_k \geq n_{k-1}$  ($n_0 =2$), and generate sample $\{\tilde{\xi}^1,\ldots,\tilde{\xi}^{n_k}\}$} \label{step21}
	    \State{\textsc{Assess Solution Quality:}\ \ Calculate \texttt{GAP}$_{n_k}$($\xh_k$) and \texttt{SV}$_{n_k}$($\xh_k$)} \label{step22}
		\If{\texttt{GAP}$_{n_k}$($\xh_k$) $\leq$ $h'\sqrt{\texttt{SV}_{n_k}(\xh_k)}+\varepsilon'$} \label{step3}
	    	\State{Set $T=k$}
			\State{\textsc{Output:}\ \ Solution $\xh_T=\xh_k$ with \texttt{CI}$_{n_T}(\xh_T)=\left[0,h\sqrt{\texttt{SV}_{n_T}(\xh_T)}+\varepsilon\right]$} \label{output}
	    	\State{STOP = Yes}
	    \EndIf
    \EndWhile
    \EndProcedure
  \end{algorithmic}
}
\end{algorithm}
\vspace*{-0.05in}

\section{Sequential Sampling with Variance Reduction}
\label{sec:vr}
\vspace*{-0.05in}

\subsection{Problem Class and Assumptions}\vspace*{-0.05in}
\label{ssec:problemclass}

Before we embed AV and LHS to the sequential sampling procedure, let us first narrow down the type of problems we consider. 
We make the following assumptions regarding \ref{SP} and collectively refer to them as (A$_{\text{SP}}$):\vspace*{-0.03in}
\begin{description}[0.5cm]
	\item[(A$_{\text{SP}}$)\label{(Aso)}] [{\small \it Assumptions on \ref{SP}}]\vspace*{-0.1in}
	\begin{description}[0.5cm]
		\item[(i)\label{Aso1}] $f(\cdot, \tilde{\xi})$ is continuous on $X$, w.p.1;
		\item[(ii)\label{Aso2}] There exists a positive real number $C$ such that $\sup_{x\in X, \xi \in \Xi}|f(x,\xi)|\leq C$;
		\item[(iii)\label{Aso3}] $X \neq \emptyset$ and is compact;
		\item[(iv)\label{Aso4}] Components of $\tilde{\xi}$ are independent with either analytically or numerically invertible cumulative distribution functions.
	\end{description}
\end{description}

This class of \ref{SP} includes numerous important problems like two-stage stochastic linear programs, stochastic convex programs, and stochastic discrete programs, provided they satisfy the above assumptions.   
Recall the continuity assumption \ref{(Aso)}\ref{Aso1} is automatically satisfied when $X$ is discrete.
So, two-stage stochastic integer programs with binary first-stage variables work in our framework. Their second stage could have integrality constraints.  
However, \ref{(Aso)} excludes, e.g., stochastic integer programs with mixed-integer first-stage variables and integer second-stage variables due to the non-continuity of $f$.

The boundedness condition \ref{(Aso)}\ref{Aso2} is needed to invoke CLT for LHS \citep{owen1992central}.  
The third condition may be violated for some \ref{SP}; 
however, it is not overly restrictive in practice.  
Problems with unbounded $X$ and finite optimal solutions can be restricted to a bounded subset of $X$.  
Finally, while the random parameters of many problems may have complicated dependencies, the last condition is necessary to easily apply AV and LHS.

Let us discuss in more detail when the independence assumption in \ref{(Aso)}\ref{Aso4} may be satisfied.  
First, for some problems, the independence assumption may be a natural or reasonable assumption to make. Consider a large-scale online retailer that sells a variety of products---some very different from others. A high-level strategic planning model for this retailer might consider aggregate demands for groups of products.  For example, group 1 includes toothpaste and toothbrushes, whereas group 2 includes diamond earrings and diamond bracelets.  If the groups have sufficiently different characteristics, it may be reasonable to assume the aggregate demands for groups are independent (even though demand of individual products within a group may be dependent).

The independence assumption may also be satisfied when each random parameter of \ref{SP} can be represented through simple linear or monotone transformation of mutually independent random variables. 
For example, a correlated multivariate normal vector can be written as an affine function of independent standard normal random variables.  
Other distributions that can be written using simple monotone transformations include multivariate lognormal and a special case of multivariate gamma  \citep{prekopa_szantai_78,keri_szantai_11}.  These distributions have been used in inventory management \citep{Choi2011},  finance \citep{Battocchio2007}, and water resources management \citep{dupacova_etal_91}. 
When simple monotone transformations of independent random variables are used, we assume these transformations are  embedded into the function $f(x,\tilde{\xi})$, and LHS and AV are applied to the independent components $\tilde{\xi}$. 
\added{
Observe that the independent components could include unbounded random variables like standard normal, which may lead to to violations of the boundedness assumption \ref{(Aso)}\ref{Aso2}. In practice, such components are often finitely discretized to form \ref{SP}. 
Finite discretizations typically allow the boundedness assumption to be satisfied. 
We will discuss the implication of  transformations on other conditions needed for our theoretical results in Section \ref{ssec:Asym}.
}

Finally, the above discussion highlights an important way independence arises in many practical problems, including one of our test problems \citep{sen1994network}. 
Random parameters in many applications are  forecast through statistical methods. These are then input into \ref{SP} as `fixed forecast + random error'. A typical assumption is that the errors or an underlying vector that forms the errors after  monotone transformations are independent and identically distributed,
and they form the components of $\tilde{\xi}$. 
Once again, LHS and AV can be applied to these independent components.

When the random parameters have complex dependencies that cannot be handled by the above methods, it may be possible to apply AV 
\citep{rubinstein_etal_85} and LHS \citep{stein_87,packham_schmidt_10} to the dependent components, but we do not pursue it here. We also refer the readers to \cite{bilgho:06} for methods to write components of a random vector as more complicated functions of independent variables.

We also make the following assumption:\vspace*{-0.01in} 
\begin{description}[1cm]
\item[(A1)\label{(A1)}] The sequence of candidate solutions $\{\xh_k\}$ has at least one limit point in $X^*$, w.p.1.
\end{description}
\noindent
This assumption allows us to generate the candidate solutions by any method that satisfies \ref{(A1)}. 
Many common methods satisfy this assumption under mild conditions. 
For example, independent SAAs with increasing sample size---e.g., \ref{SP_nk} with $n=n_k$---generates such a sequence for various problems \citep{shaDR:09}. 
In fact, SAAs can be generated with AV or LHS under appropriate conditions \citep{homem2008rates}.
Other sampling-based methods could be used, including stochastic approximation, stochastic decomposition for two-stage stochastic programs with recourse, and simulation optimization methods; see, e.g., the compilations \citep{hbsimopt_05,hombay_14}.

With these assumptions in place, we are ready to embed AV and LHS to the sequential framework. 
Throughout the rest of the paper, subscript $s \in \{\mathcal{I},\mathcal{A},\mathcal{L}\}$ indicates IID, AV, and LHS. 
For instance, $x^*_{\mathcal{A}, n}$ and $x^*_{\mathcal{L}, n}$ denote an optimal solution of SAA formed via AV and LHS, respectively. 
As mentioned before, in the optimality gap estimators, we also use superscript $a \in \{S,A\}$ to indicate SRP and A2RP. 
So, the pair
$(s,a)$ represents (sampling method, assessing solution quality procedure), and
 $\mathtt{GAP}^a_{s,n}(\xh)$ and  $\mathtt{SV}^a_{s,n}(\xh)$ denote the optimality gap estimators formed with sampling method $s$ and assessment procedure $a$ for a given $\xh \in X$.
Similarly, from this point on, when we refer to the sequential sampling procedure summarized in Algorithm \ref{alg:seq} or the stopping rule (\ref{T}), we assume a specific ($s$, $a$) pair.
\vspace*{-0.05in}

\subsection{Sequential Sampling with AV}
\label{ssec:av}

Let us first recall how to sample an AV pair: (i) Draw an independent uniform variate $\tilde{u}_\mathcal{A}$ from a $U(0, 1)^{d_{\xi}}$ distribution (recall $d_\xi$ denotes the dimension of $\tilde{\xi}$); then, (ii) obtain the AV pair  $\tilde{u}_{\mathcal{A}'}=\overrightarrow{1}-\tilde{u}_\mathcal{A}$, where $\overrightarrow{1}$ is a vector of 1s; finally (iii) apply the inverse transform method to $(\tilde{u}_{\mathcal{A}}, \tilde{u}_{\mathcal{A}'})$ to obtain $(\tilde{\xi}_{\mathcal{A}}, \tilde{\xi}_{\mathcal{A}'})$.
Each AV pair is independent and identically distributed but the observations within pairs are not.

To embed AV to sequential sampling, we must update the optimality gap estimators. 
Let $g_{\mathcal{A}}(x,\tilde{\xi})=\frac{1}{2}\big(f(x,\tilde{\xi}_{\mathcal{A}})+f(x,\tilde{\xi}_{\mathcal{A}'})\big)$, where  $(\tilde{\xi}_{\mathcal{A}},\tilde{\xi}_{\mathcal{A}'})$ is an AV pair. 
All AV estimators are defined using this quantity. 
For instance, SAA now becomes $\min_{x \in X}\frac{2}{n} \sum_{i=1}^{n/2} g_{\mathcal{A}}(x,\tilde{\xi}^i)$; compare to \ref{SP_nk}. 
Similarly, the SRP estimators using AV are updated to 
\begin{small}
\[
\texttt{GAP}^{S}_{\mathcal{A},n}(\xh)=\frac{\displaystyle{\sum^{n/2}_{i=1}}\left[g_{\mathcal{A}}(\xh,\tilde{\xi}^i)-g_{\mathcal{A}}(x^*_{\mathcal{A}, n},\tilde{\xi}^i)\right]}{n/2},\ \ \ 
\texttt{SV}^{S}_{\mathcal{A},n}(\xh)= 
\frac{\displaystyle{\sum^{n/2}_{i=1}}\left[\left(g_{\mathcal{A}}(\xh,\tilde{\xi}^i)-g_{\mathcal{A}}(x^*_{\mathcal{A}, n},\tilde{\xi}^i)\right)-\texttt{GAP}^{S}_{\mathcal{A},n}(\xh)\right]^2}{n/2-1}.
\]
\end{small}%
As in IID, A2RP AV averages two independent SRP AV estimators. Hence they are omitted for brevity.

Another change occurs in the sample-size increase formula.   
Iteration $k$ of the sequential procedure with AV generates $n_k/2$ independent AV pairs $\left\lbrace \tilde{\xi}^1_{\mathcal{A}},\tilde{\xi}^{1}_{\mathcal{A}'},\ldots,\tilde{\xi}^{n_k/2}_{\mathcal{A}},\tilde{\xi}^{n_k/2}_{\mathcal{A}'}\right\rbrace$. Thus, we should have an even number of observations for SRP.  
Similarly, because A2RP uses two SRP estimators, $n_k$ should be divisible by four when using A2RP with AV.  
Finally, we modify the sample-size increase formula given in \eqref{nk_IID} to have $n_k/2$ on its left side. 
This is needed to invoke the IID theory on each AV pair; see Section \ref{sec:the}. 
But it also causes some changes in the behavior of AV sequential procedures; 
we will discuss {these} in Section \ref{sec:exp}. 
Table~\ref{tb: size} summarizes the necessary adjustments in the sample sizes.

\subsection{Sequential Sampling with LHS}
\label{ssec:lhs}

Let us briefly review LHS. 
To form an LHS sample of size $n_k$, denoted $\left\{\tilde{\xi}^1_{\mathcal{L}},\ldots,\tilde{\xi}^{n_k}_{\mathcal{L}}\right\}$, at iteration $k$ of the sequential procedure, for each dimension $j=1,\ldots,d_{\xi}$ (i) first, sample  $\tilde{u}^i_j$ from $U\big(\frac{i-1}{n_k},\frac{i}{n_k}\big)$ for $i=1,\ldots,n_k$; then (ii) define a random permutation of $\{1,\ldots,n_k\}$ as $\{\pi^1_j,\ldots,\pi^{n_k}_j\}$, where each of the $n_k$! permutations is equally likely; and finally (iii) apply the inverse transform method to $\tilde{u}^{\pi^i_j}_j$ to obtain $\tilde{\xi}^i_{\mathcal{L},j}$ for $i=1,\ldots,n_k$.
LHS, like AV, induces negative correlations to reduce variance. It also employs stratification on each component of $\tilde{\xi}$ to sample more evenly than IID. 

The LHS optimality gap estimators are formed similar to the IID ones given in \eqref{eq:srp} and \eqref{eq:a2rp}, with the appropriate changes in the notation and the sampling method used. 
The sample-size increase formula also remains the same; see Table~\ref{tb: size}.

\begin{table}[!tbh]
\centering
\caption{Sample Sizes used in Sequential Procedures}
\label{tb: size}
{\small
\begin{tabular}{cccc} \toprule
Sampling  method     & Assessing procedure & Sample size            & $n_k$                                                     \\ \cmidrule{1-4}
IID \& LHS           & SRP / A2RP      & -- / even                & $n_k \geq \left(\frac{1}{h-h'}\right)^2(c_p+2p\ln^2k)$            \\
AV                   & SRP / A2RP      & even / multiple of 4     & $\frac{n_k}{2}  \geq \left(\frac{1}{h-h'}\right)^2(c_p+2p\ln^2k)$            \\ \cmidrule{1-4}
\end{tabular}}
\end{table}

We note one difference when using LHS within the sequential sampling procedure. 
The sequential framework allows creating a new, independent sample at each iteration or appending additional observations to the existing ones. 
While AV allows any type of sampling (appended or entirely new AV pairs or mixture like appending AV pair of scenarios for a number of iterations followed by an entirely new sample), this is not true of LHS. 
Increasing the LHS sample size requires a completely new LHS sample. 
Therefore, LHS sequential procedure generates a new set of observations at each iteration.

LHS requires essential changes to the {assumptions and} proofs. We will discuss these in the next section.
\vspace*{-0.05in}

\section{Theoretical Analysis} \vspace*{-0.05in}
\label{sec:the}

We now examine the theoretical properties of the proposed procedures.   
Before we begin the analysis, let us introduce some results and properties used in the proofs.
\vspace*{-0.05in}

\subsection{Preliminaries}
\label{ssec:prelim}
\vspace*{-0.05in}

Assumptions \ref{(Aso)}\ref{Aso1}--\ref{Aso3} ensure that \ref{SP} has a finite optimal solution; so $X^*$ is nonempty. 
Similarly, all SAAs---formed with any sampling method---have finite optimal solutions, w.p.1. 
The asymptotic properties of the optimality gap estimators with IID, AV, and LHS have been studied in the non-sequential setting by, e.g., \cite{bayraksan2006assessing,drew2007quasi,Stockbridge2013Bias}. 
These studies show that, under \ref{(Aso)}, the optimality gap estimators are strongly consistent.  
For instance, $\texttt{GAP}^{a}_{s,n}(\xh) \rightarrow \mathcal{G}_{\xh}$ as $n\rightarrow \infty$, w.p.1 for all sampling methods $s=\mathcal{I}, \mathcal{A},\mathcal{L}$ using any assessment procedure SRP or A2RP ($a=A,S$).

For the theoretical analysis of sequential sampling, we need a somewhat stronger result for the \texttt{GAP} estimators---albeit we need this condition to hold in probability, not w.p.1.
Specifically, let $n_k$ satisfy $n_k \to \infty$ as $k \to \infty$, and let $\xh_k \in X$ be a sequence of candidate solutions with $x$ as one of its limit points, i.e., there exists a subsequence $\mathcal{K}$ of $\{1,2,3,\ldots\}$ such that $\displaystyle\lim_{k \rightarrow \infty,\ k \in \mathcal{K}}\xh_k = x$. 
    Then, for any $(s,a)$ pair and $\delta >0$, 
\begin{equation}
\label{eq:uniformconv}
\liminf_{k \to \infty} \p{|\texttt{GAP}^a_{s, n_k}(\xh_k)-\mathcal{G}_{x})|> \delta}=0. 
\end{equation}
Given assumption \ref{(Aso)}, a sufficient condition to ensure \eqref{eq:uniformconv} is that a {\it uniform} Strong Law of Large Numbers (SLLN) holds. 
This means that, in the case of LHS, $\frac{1}{n}\sum_{i=1}^{n}f(x,\tilde{\xi}_{\mathcal{L}}^{i})$---and, in the case of AV, $\frac{2}{n}\sum_{i=1}^{n/2}g_{\mathcal{A}}(x,\tilde{\xi}^{i})$---converges to $\e{f(x,\tilde{\xi})}$ uniformly in $X$.

The uniform SLLN can be obtained for stochastic programs satisfying \ref{(Aso)} by invoking Proposition 7 of \citet{shapiro2003monte}. 
Conditions \ref{(Aso)}\ref{Aso1}--\ref{Aso3} satisfy the requirements of this proposition, but in the proof of that proposition, we must use a pointwise SLLN for AV and LHS instead of IID. 
Pointwise SLLN holds for AV by considering two IID subsets  $\left\lbrace \tilde{\xi}^1_{\mathcal{A}},\ldots,\tilde{\xi}^{n/2}_{\mathcal{A}}\right\rbrace$ and $\left\lbrace \tilde{\xi}^{1}_{\mathcal{A}'},\ldots,\tilde{\xi}^{n/2}_{\mathcal{A}'}\right\rbrace$, 
and it holds for LHS under the boundedness assumption of \ref{(Aso)}\ref{Aso2} \citep{Loh1996Latin}.

The consistency of the variance estimator is slightly different.  
To see this, observe that the optimality gap of $\xh \in X$ can be written as $\mathcal{G}_{\xh} =\e{f(\xh,\tilde{\xi}) - f(x^*, \tilde{\xi})}$ for any $x^* \in X^*$. 
The value of $\mathcal{G}_{\xh}$ remains the same for {\it all} optimal solutions $x^* \in X^*$ for a given $\xh \in X$. 
However, the variance $\var{f(\xh,\tilde{\xi}) - f(x^*, \tilde{\xi})}$ might change for different optimal solutions $x^* \in X^{*}$. 
To correctly handle the consistency of the sample variance estimators, for a given $\xh \in X$, define 
\begin{equation}  
\sigma^2_{s}(\xh)=\left\{
    \begin{matrix}
\var{f(\xh,\tilde{\xi})-f(x^*_{s,\min},\tilde{\xi})} & \text{if }s=\mathcal{I}, \mathcal{L}
\\
\var{g_\mathcal{A}(\xh,\tilde{\xi})-g_\mathcal{A}(x^*_{\mathcal{A},\min},\tilde{\xi})} & \text{if }s=\mathcal{A}, 
            \label{eq: true}
    \end{matrix}
\right.
\end{equation}
where $x^*_{s,\min} \in \arg\min_{y\in X^*}\var{f(\xh,\tilde{\xi})-f(y,\tilde{\xi})}$ if $s=\mathcal{I}, \mathcal{L}$ and $x^*_{\mathcal{A},\min} \in \arg\min_{y\in X^*}\var{g_\mathcal{A}(\xh,\tilde{\xi})-g_\mathcal{A}(y,\tilde{\xi})}$ if $s=\mathcal{A}$. 
Above, we suppress the dependence of $\xh$ on $x^*_{s,\min}$ for simplicity. 
Similarly, define $\sigma^2_{s,\max}(\xh)$ for each sampling method $s=\mathcal{I},\mathcal{A},\mathcal{L}$.
For brevity, we only formally define it for IID and LHS:  
\begin{equation}
    \label{eq:sigma_max}
    \sigma^2_{s,\max}(\xh)=\var{f(\xh,\tilde{\xi})-f(x^*_{s,\max},\tilde{\xi})}, 
\end{equation}
where $x^*_{s,\max} {\in} \mathrm{argmax}_{y\in X^*}\var{f(\xh,\tilde{\xi})-f(y,\tilde{\xi})}$ for $s=\mathcal{I},\mathcal{L}$. 
The variance  $\sigma^2_{\mathcal{A},\max}(\xh)$ for AV is defined similarly using the preceding construct. 

With the above definitions, under \ref{(Aso)}, the consistency of the sample variance estimator is given by 
\begin{equation}
\label{eq:sv}
\sigma^2_{s}(\xh) \le 
\lim \inf_{n\rightarrow \infty} \texttt{SV}^{a}_{s,n}(\xh)  \le \lim\sup_{n\rightarrow \infty} \texttt{SV}^{a}_{s,n}(\xh) \le 
\sigma^2_{s,\max}(\xh),
\end{equation}
w.p.1, for all $\xh \in X$, using any assessment procedure SRP or A2RP ($a=A,S$) and any sampling method $s=\mathcal{I}, \mathcal{A},\mathcal{L}$. 
When \ref{SP} has a unique optimal solution $X^{*}=\{x^{*}\}$, the above consistency result reverts to the usual case:  $\texttt{SV}^{a}_{s,n}(\xh) \rightarrow \sigma^2_{s}(\xh)$, w.p.1 as $n \rightarrow \infty$ for all $(s,a)$ pairs.  
Analyzing the sequential procedures requires a stronger version of consistency given in \eqref{eq:uniformconv} for the \texttt{GAP} estimators but only a regular version of consistency given in \eqref{eq:sv} for the \texttt{SV} estimators.

Before we present the theoretical results, we introduce another set of optimality gap estimators and discuss their properties. These will be used in the proofs. 
We refer to these estimators as `non-optimized' counterparts of the \texttt{GAP} estimators because the \texttt{GAP} estimators require optimization of an SAA problem (see the second term in \eqref{eq:gapsrp}), but the estimators \added{below} do not. Instead, they directly use an optimal solution. We will further discuss the differences between the \texttt{GAP} estimators and their non-optimized counterparts after their definition.

For SRP, given $\xh \in X$, let 
\begin{align}
\label{eq:Ds}
D^S_{s,n}(\xh)=\left\{\begin{matrix}
\frac{1}{n} \sum_{i=1}^{n}\left(f(\xh,\tilde{\xi}^i_s) - f(x^*_{s,\min}, \tilde{\xi}^i_s)\right)       & \text{if }s=\mathcal{I}, \mathcal{L}
\hspace{0.3cm}
\\ 
\frac{2}{n} \sum_{i=1}^{n/2}\left(g_\mathcal{A}(\xh,\tilde{\xi}^i) - g_\mathcal{A}(x^*_{\mathcal{A},\min}, \tilde{\xi}^i)\right) & \text{if }s=\mathcal{A},
\end{matrix}\right.
\end{align}
where above we use $\tilde{\xi}_{\mathcal{I}}\equiv \tilde{\xi}$ for notational convenience.
Recall that elsewhere we simply use $\tilde{\xi}$ within $f(\cdot,\tilde{\xi})$---instead of $\tilde{\xi}_{\mathcal{I}}$---for IID.
The A2RP non-optimized optimality gap estimator is simply the average of two independent SRP non-optimized estimators: 
$$D^A_{s,n}(\xh)=\frac{D^{S,1}_{s,n/2}(\xh)+D^{S,2}_{s,n/2}(\xh)}{2},\  s=\mathcal{I}, \mathcal{A}, \mathcal{L}.$$
For each $(s,a)$ pair, we assume the same set of observations are used to form both $D^a_{s,n}(\xh)$ and $\texttt{GAP}^a_{s,n}(\xh)$ for a fixed $\xh \in X$. 

Observe the differences between $\texttt{GAP}^{S}_{s,n}(\xh)$ given, e.g., in \eqref{eq:gapsrp} for the IID case and $D^S_{s,n}(\xh)$ given in \eqref{eq:Ds}:  
$\texttt{GAP}^{S}_{s,n}(\xh)$ uses an optimal solution to SAA $x^*_{s,n}$, which is a random variable because it depends on the whole sample.  
In contrast, $D^S_{s,n}(\xh)$ uses $x^*_{s,\min}$, which is an unknown but {\it fixed} (non-random) optimal solution. 
Because both $\xh$ and $x^*_{s,\min}$ are fixed, $D^S_{s,n}(\xh)$---unlike \texttt{GAP}---is a regular sample mean. 
As such, $D$ is unbiased for all $(s,a)$ pairs: $\e{D^{a}_{s,n}(\xh)}=\mathcal{G}_{\xh}$. 
Furthermore, for all  $\xh \in X$, $n \ge 2$, and $(s,a)$ pairs, we have\vspace*{-0.05in} 
\begin{equation}
\label{eq:GtoD}
\texttt{GAP}^{a}_{s,n}(\xh)\geq D^{a}_{s,n}(\xh), 
\end{equation}
w.p.1.
The above inequality is a consequence of  $x^*_{s,\min}$ being suboptimal (or at best optimal) to the SAA problem solved to form $\texttt{GAP}^{a}_{s,n}(\xh)$.
Let us illustrate by SRP IID; see \eqref{eq:gapsrp}. 
Because $z_{\mathcal{I},n}^*$ in \eqref{eq:gapsrp} is the optimal value, for any $y \in X$, we have $z_{\mathcal{I},n}^* \leq \frac{1}{n}\sum_{i=1}^{n} f(y,\tilde{\xi}^i)$. 
Replacing $y$ in this inequality 
by $x^*_{\mathcal{I}, \min} \in X^*$, we obtain \eqref{eq:GtoD} for $(s,a)={\mathcal{I},S}$. 
Other pairs follow by similar arguments with appropriate changes to notation.

Finally, the sequential proofs rely on CLT. 
That is, assuming finite and non-zero variances, for any $(s,a)$ pair and any $\xh \in X$,\vspace*{-0.05in} 
\begin{equation}
\label{eq:clt}
\left(D^{a}_{s,n}(\xh)-\mathcal{G}_{\xh}\right)\Big/\sqrt{\var{D^{a}_{s,n}(\xh)}} \Rightarrow N(0,1)\text{ as } n\to \infty,
\end{equation}
where $\Rightarrow$ denotes convergence in distribution.  
The above CLT holds for AV by reverting to the IID case, 
and it holds for LHS by the boundedness assumption \ref{(Aso)}\ref{Aso2}  \citep{owen1992central}. 

In \eqref{eq:clt}, we assumed finite, non-zero variances.
Finiteness of the variance term is guaranteed by \ref{(Aso)}\ref{Aso2}.
However, it may happen that $\var{D^{a}_{s,n}(\xh)} = 0, \forall n$. 
For instance, this could happen when \ref{SP} has a unique optimal solution $x^*$ and $\xh=x^*$. 
This yields $\mathcal{G}_x=\sigma^2_s(x)=0$ and $D^{a}_{s,n}(\xh)=0$, w.p.1 for all sampling methods. 
Or, in the case of LHS, if $f(\xh,\cdot)$ and its corresponding $f(x^*_{\mathcal{L},\min},\cdot)$  are additive functions---to be defined precisely in Section \ref{ssec:Asym}---we have $n\cdot\var{D^{a}_{\mathcal{L},n}(\xh)} \rightarrow 0$ as $n \rightarrow 0$. 
Then, a CLT in the form of \eqref{eq:clt} cannot be used for LHS. 
But $\sqrt{n}\left(D^{a}_{\mathcal{L},n}(\xh)-\mathcal{G}_{\xh}\right)$ converges in distribution, and hence in probability, to zero in this case \citep{stein_87}. 
Fortunately, our analysis remains unaffected by these `degenerate' cases; see, e.g., Step 2 of the proof in Section \ref{ssec:Asym}.

\subsection{Finite Stopping and Asymptotic Validity under Sublinear Sample Sizes} 
\label{ssec:Asym}

We first show conditions under which the sequential sampling procedures with AV and LHS stop in a finite number of iterations and produce asymptotically valid confidence intervals using the sample-size increases of Table \ref{tb: size}, with $c_p$ defined in Section \ref{ssec:overview}. 
Table \ref{tb: size} indicates that the sample sizes must be increased at least at a sublinear rate in the iteration number $k$, specifically of order $O(\ln^2 k)$. 
Similar results have been shown under an assumption that is not straightforward to establish for LHS (but relatively simple for IID) 
\citep{bayraksan2011sequential}.
So, we bypass this assumption and directly work with another assumption based on known properties of LHS large deviation probabilities. Proofs must be adjusted to accommodate LHS and this alternative assumption.

For any $\theta \in \mathbb{R},$ let\vspace*{-0.03in}
\begin{align*}
\Phi_{\mathcal{L}, n}(\xh,\theta)&=\e{\exp\left(\theta \sum^{n}_{i=1}\left(f(\xh,\tilde{\xi}^i_{\mathcal{L}})-\e{f(\xh,\tilde{\xi})}\right)\right)}, \text{ and} \\
\Phi_{\mathcal{I}, n}(\xh,\theta)&=\e{\exp\left(\theta\left(f(\xh,\tilde{\xi})-\e{f(\xh,\tilde{\xi})}\right)\right)}^n.
\end{align*}
In the first quantity $\Phi_{\mathcal{L}, n}(\xh,\theta)$, we could have equivalently used $\e{f(\xh,\tilde{\xi}^i_{\mathcal{L}})}$, $i=1,2,\ldots,n$. 
Because each observation in the LHS sample has distribution $P$ \citep{mckay1979comparison},  LHS produces unbiased estimators and
we have $\e{f(\xh,\tilde{\xi}^i_{\mathcal{L}})}=\e{f(\xh,\tilde{\xi})}$ for all $i=1,2,\ldots,n$.

Consider the following assumption for LHS:
\begin{description}[1cm]
\item[(A2)\label{(A2)}]  For all $\theta \in \mathbb{R}$, $\xh\in X$ and $n\geq 2$, $\Phi_{\mathcal{L}, n}(\xh,\theta)\leq \Phi_{\mathcal{I}, n}(\xh,\theta)$.  
\end{description}
Assumption \ref{(A2)} implies the upper bound on large deviation probabilities can go to zero faster for LHS than for IID by increasing $n$ \citep{drew2012some}.
That is, given $\xh \in X$, for any $l< 0 < u$, \ref{(A2)} implies
\begin{small}
\begin{align*}
\p{\frac{1}{n}\sum^{n}_{i=1}\left(f(\xh,\tilde{\xi}^i_{\mathcal{L}})-\e{f(\xh,\tilde{\xi})}\right) \leq l} &\leq \exp\left(-nI_{\mathcal{L},n}(\xh,l)\right)\leq \exp\left(-nI_{\mathcal{I}}(\xh,l)\right), 
\\
\p{\frac{1}{n}\sum^{n}_{i=1}\left(f(\xh,\tilde{\xi}^i_{\mathcal{L}})-\e{f(\xh,\tilde{\xi})}\right) \geq u} &\leq \exp\left(-nI_{\mathcal{L},n}(\xh,u)\right)\leq \exp\left(-nI_{\mathcal{I}}(\xh,u)\right),
\end{align*}
\end{small}%
where $I_{\mathcal{L},n}(\xh,z):=\sup_{\theta \in \mathbb{R}}\left[\theta z - \frac{1}{n}\log \Phi_{\mathcal{L},n}(\xh,\theta) \right]$ and $I_{\mathcal{I}}(\xh,z):=\sup_{\theta \in \mathbb{R}}\left[\theta z - \frac{1}{n}\log \Phi_{\mathcal{I}, n}(\xh,\theta) \right]$.
In the sequential sampling framework, we want to show deviation probabilities similar to above are bounded---but over all possible candidate solutions $\xh_k$ and iterations $k$ because both the candidate solution and the iteration when the procedures stop are random variables. 
We use specific values of $\theta$ for this purpose.
This can be seen through equations \eqref{ineq: P(A+B<C)}--\eqref{ineq: (A2)LHS} in the proof of Theorem \ref{thm: cp}.

When is assumption \ref{(A2)} satisfied? 
We narrow our focus to problems satisfying \ref{(Aso)} and provide some examples.\vspace*{-0.05in} 
\begin{enumerate}
\item First, if $d_{\xi}=1$ and for all $\xh \in X$ $f(\xh,\cdot)$ is continuous on $\Xi$ or the set of points at which $f(\xh,\cdot)$ is discontinuous has Lebesgue measure zero, \ref{(A2)} is satisfied. 
For instance, under mild conditions, a traditional single-product newsvendor problem fits this case.

\item When $d_{\xi}>1$, if for all $\xh \in X$ $f(\xh,\cdot)$ is an additive function---i.e., $f(\xh,\cdot)$ w.p.1 can be written as 
$f(\xh,\tilde{\xi})=K+\sum^{d_{\xi}}_{j=1}f_j(\xh,\tilde{\xi}_j)$, where $K$ is a constant and $f_j(\xh,\cdot)$, $j=1,\ldots, d_{\xi}$ are univariate functions---then $f$ satisfies \ref{(A2)} by extension of the one-dimensional case. Here, we assume $f_j(\xh,\cdot), j=1,\ldots,d_{\xi}$ are continuous or the set of points at which they are discontinuous has Lebesgue measure zero. 
This continuity assumption is satisfied, for instance, if $\tilde{\xi}_j$ are finite discrete random variables. 

\item Otherwise---when $d_{\xi}>1$ but $f$ is not additive---$f(\xh,\tilde{\xi})$ that is monotone in each component of $\tilde{\xi}$ (while other components are fixed) satisfies \ref{(A2)}.
Monotone functions preserve the negative dependence of LHS observations \citep{jin2003probabilistic}, and consequently \ref{(A2)} holds.
\end{enumerate}

A well-known problem class with monotone functions is two-stage stochastic linear programs with fixed recourse, given by\vspace*{-0.05in}
\begin{align}
&\min_{x} \left\lbrace cx+\e{h(x,\tilde{\xi})} : Ax=b, x\geq 0\right\rbrace, \label{SLP2}
\end{align}
where $h(x,\tilde{\xi})=\min_y \left\lbrace \tilde{q}y : Wy\leq \tilde{R}-\tilde{T}x, y \geq 0\right\rbrace$ and $\tilde{\xi}$ is comprised of random elements of $\tilde{q}, \tilde{R}$ and $\tilde{T}$.
The objective function of (\ref{SLP2}), written as $f(\xh,\tilde{\xi})= c\xh+{h(\xh,\tilde{\xi})}\text{ for each }\xh \in X$, is monotone in each component of $\tilde{\xi}$ because components of $\tilde{\xi}$ are independent, $W$ is non-random (fixed recourse), $\xh\geq 0, y\geq 0$, and the second-stage constraints $\left(Wy\leq \tilde{R}-\tilde{T}\xh\right)$ are inequalities \citep{higle1998variance,bailey1999response, drew2012some}.

When the random parameters of \ref{SP} are functions of $\tilde{\xi}$, the situation is more delicate. For simplicity, assume for the above problem (\ref{SLP2}) that $q$ and $T$ are fixed and components of vector $\tilde{R}$ are functions of $\tilde{\xi}$. We denote this as $R(\tilde{\xi})$. Then, $h(x,\tilde{\xi})$ is updated to $h(x,\tilde{\xi})=\min_y \big\{ qy : Wy\leq R(\tilde{\xi})-Tx,\ \ y \geq 0\big\}$.  
Instead of $R(\tilde{\xi})=\tilde{\xi}$ discussed above, suppose now $R(\tilde{\xi})$ is defined through $R_j(\tilde{\xi})=\gamma_j \tilde{\xi}_0 + \tilde{\xi}_j$ for each constraint $j=1,2,\ldots,m_2$, where $\tilde{\xi}_0, \tilde{\xi}_1, \tilde{\xi}_2,\ldots, \tilde{\xi}_{m_2}$ are independent random variables. Here, the dimension of $\tilde{\xi}$ is $d_{\xi}=m_2 +1$. 
If $\gamma_j > 0$ for all $j=1,2,\ldots,m_2$, increasing $\tilde{\xi}_0$ (while all other random parameters are fixed) results in all constraints to be more relaxed. This means that the optimal value either only decreases or stays the same. As a result, we have monotonicity in $\tilde{\xi}_0$. 
However, if some $\gamma_j$ were positive while others were negative, we couldn't have concluded monotonicity in $\tilde{\xi}_0$. This is because some constraints might become more relaxed while others become tighter, and so we cannot determine the direction of the optimal value. 

\smallskip 

The theorem below establishes finite stopping and asymptotic validity of the sequential procedures with AV and LHS. 

\begin{thm}\label{thm: cp}
	Suppose \ref{(A1)} and \ref{(Aso)} hold.
	For LHS, additionally suppose \ref{(A2)} holds.
	Let $\varepsilon > \varepsilon' > 0$, $h'>0$, $p > 0$, and $0 < \alpha < 1$ be fixed.
	Consider the sequential sampling procedure where the sample sizes are increased according to Table~\ref{tb: size}.  
	If the procedure stops at iteration $T$ according to (\ref{T}), then for all $(s, a)=(\mathcal{A}, S), (\mathcal{A}, A), (\mathcal{L}, S), (\mathcal{L}, A)$,
\begin{enumerate}
\item[(i)]  $P(T < \infty) = 1$,
\item[(ii)] $\liminf_{h \downarrow h'} \p{\mathcal{G}_{\xh_T} \leq h\sqrt{\texttt{SV}^{a}_{s,n_T}(\xh_T)} + \varepsilon} \geq  1- \alpha$.
\end{enumerate}
\end{thm}
\smallskip 

\begin{proof}
{\it (i)} Under \ref{(A1)}, the proof follows the proof of \cite[Proposition 1]{bayraksan2011sequential} by noting that \eqref{eq:uniformconv} holds for all $(s,a)$ pairs.

\noindent
{\it (ii)}
Let $\Delta h = h - h'$ and $\Delta \varepsilon = \varepsilon - \varepsilon'$. 
For any given $(s,a)$ pair, we have 
$\p{\mathcal{G}_{\xh_T} > h\sqrt{\texttt{SV}^{a}_{s,n_T}(\xh_T)} + \varepsilon} \leq \sum_{k=1}^{\infty} \p{D^{a}_{s,n_{k}}(\xh_{k}) - \mathcal{G}_{\xh_k} \leq -\Delta h \, \sqrt{\texttt{SV}^{a}_{s,n_{k}}(\xh_{k})} - \Delta \varepsilon}$ 
by using the stopping rule \eqref{T} and the inequality $\texttt{GAP}^{a}_{s,n}(\xh)\geq D^{a}_{s,n}(\xh)$ in \eqref{eq:GtoD}.
So, it suffices to show
\begin{align}\label{ineq: Fatou}
\limsup_{\Delta h \downarrow 0} \sum_{k=1}^{\infty} \p{D^{a}_{s,n_{k}}(\xh_{k}) - \mathcal{G}_{\xh_k} \leq  -\Delta h \, \sqrt{\texttt{SV}^{a}_{s,n_{k}}(\xh_{k})} - \Delta \varepsilon} \leq \alpha
\end{align}
for all  $(s, a)=(\mathcal{A}, S), (\mathcal{A}, A), (\mathcal{L}, S), (\mathcal{L}, A)$.
The rest of the proof consists of two main steps to show inequality (\ref{ineq: Fatou}). 
Step 1 shows boundedness of the infinite summation in (\ref{ineq: Fatou}) to apply Fatou's Lemma for sufficiently small $\Delta h$. 
Step 2 applies Fatou's Lemma, exchanging $\lim \sup$ and the infinite summation in \eqref{ineq: Fatou}, and verifies the asymptotic validity of the resulting confidence interval. 
Below, we provide a detailed proof for SRP LHS, i.e., $(s,a)=(\mathcal{L},S)$, and mention any necessary changes for the other pairs.
\medskip 

\noindent {\bf \underline{Step 1.}}
Consider the sequential procedure that uses SRP with LHS.
First, observe that although \ref{(A2)}---which will be used shortly---holds for $f$, an analogous result may not hold for $D$. 
To see this, suppose $f(\xh,\tilde{\xi})$ is monotone in each component of $\tilde{\xi}$ for all $\xh\in X$. 
Although $f(\xh,\tilde{\xi})$ is monotone, because the difference of monotone functions is not necessarily monotone, $D^{S}_{\mathcal{L},n}(\xh)$ may not be. 
Therefore, we split $D^{S}_{\mathcal{L},n}(\xh)$ into two and consider each (e.g., monotone) part separately through a probability inequality. 
For all $\gamma >0$, the left-hand side of \eqref{ineq: Fatou} is bounded above by 
\begin{small}
\begin{align}
&\sum_{k=1}^{\infty} \p{D^{S}_{\mathcal{L},n_{k}}(\xh_{k}) - \mathcal{G}_{\xh_k} \leq - \Delta \varepsilon}\nonumber\\
\leq&  \sum_{k=1}^{\infty} \int_{\xh_k} \p{\frac{1}{n_k}\sum^{n_k}_{i=1}f(\xh_k,\tilde{\xi}^i_{\mathcal{L}})-\e{f(\xh_k,\tilde{\xi})}\leq - \frac{\Delta \varepsilon}{2} \Bigg|\xh_k}dP_{\xh_k} 
\nonumber 
\\
&+\sum_{k=1}^{\infty}  \int_{\xh_k}  \p{\frac{1}{n_k}\sum^{n_k}_{i=1}f(x^*_{\mathcal{L},\min},\tilde{\xi}^i_{\mathcal{L}})-\e{f(x^*_{\mathcal{L},\min},\tilde{\xi})}\geq \frac{\Delta \varepsilon}{2}\Bigg|\xh_k}dP_{\xh_k}
\label{ineq: P(A+B<C)}
\\
\leq& \sum^{\infty}_{k=1}\int_{\xh_k}\e{\exp\left(-\frac{\gamma}{\sqrt{n_k}}\sum^{n_k}_{i=1}\left(f(\xh_k,\tilde{\xi}^i_{\mathcal{L}})-\e{f(\xh_k,\tilde{\xi})}\right)\right)\Bigg| \xh_k}dP_{\xh_k}\exp\left(-\gamma\frac{\Delta \varepsilon \sqrt{n_k}}{2}\right)   \nonumber 
\\
&+\sum^{\infty}_{k=1}\int_{\xh_k}\e{\exp\left(\frac{\gamma}{\sqrt{n_k}}\sum^{n_k}_{i=1}\left(f(x^*_{\mathcal{L},\min},\tilde{\xi}^i_{\mathcal{L}})-\e{f(x^*_{\mathcal{L},\min},\tilde{\xi})}\right)\right)\Bigg| \xh_k}dP_{\xh_k}\exp\left(-\gamma\frac{\Delta \varepsilon \sqrt{n_k}}{2}\right)  \label{ineq: chernoff1} 
\\
\leq&  \sum^{\infty}_{k=1}\sup_{x \in X}\Phi_{\mathcal{L}, n_k}\left(x,-\frac{\gamma}{\sqrt{n_k}}\right)\exp\left(-\gamma\frac{\Delta \varepsilon \sqrt{n_k}}{2}\right)+\sum^{\infty}_{k=1}\sup_{x \in X}\Phi_{\mathcal{L}, n_k}\left(x,\frac{\gamma}{\sqrt{n_k}}\right)\exp\left(-\gamma\frac{\Delta \varepsilon \sqrt{n_k}}{2}\right)  \label{ineq: intP=1} 
\\
\leq& \sum^{\infty}_{k=1}\sup_{x \in X}\Phi_{\mathcal{I}, n_k}\left(x,-\frac{\gamma}{\sqrt{n_k}}\right)\exp\left(-\gamma\frac{\Delta \varepsilon \sqrt{n_k}}{2}\right)   + \sum^{\infty}_{k=1}\sup_{x \in X}\Phi_{\mathcal{I}, n_k}\left(x,\frac{\gamma}{\sqrt{n_k}}\right)\exp\left(-\gamma\frac{\Delta \varepsilon \sqrt{n_k}}{2}\right) \label{ineq: (A2)LHS} 
\\
\leq&  2\exp\left({2(\gamma C)^2}\right)\sum^{\infty}_{k=1}k^{-\gamma\left(\frac{\Delta \varepsilon \sqrt{p}}{\sqrt{2}\Delta h}\right)},  \label{ineq: boundedLHS}
\end{align}
\end{small}where $P_{\xh_k}$ denotes the distribution function of ${\xh_k}$. 

Inequality (\ref{ineq: P(A+B<C)}) is a consequence of $P(A + B \leq c) \leq P(\{A \leq c/2\} \cup \{B \leq c/2\}) \leq P(A \leq c/2) + P(B \leq c/2)$ for two random variables $A$ and $B$ and constant $c$. 
Inequality (\ref{ineq: chernoff1}) follows from an application of the Chernoff bound and splitting $n_k$ to two $\sqrt{n_k}$s.
Inequality (\ref{ineq: intP=1}) takes supremum over $X$ and leaves out $\int_{\xh_k}dP_{\xh_k}=1$.
Inequality (\ref{ineq: (A2)LHS}) invokes \ref{(A2)}.
An application of Hoeffding's lemma \citep{hoeffding1963probability} yields an upper bound on $\sup_{x \in X} \Phi_{\mathcal{I}, n}(x,\theta)$ with $\theta = \gamma/\sqrt{n}$ or $\theta = -\gamma/\sqrt{n}$ as $\exp\left(2(\gamma C)^2\right)$. 
Recall that $C$ is defined in \ref{(Aso)}\ref{Aso2}. 
Finally, the right-hand side of (\ref{ineq: boundedLHS}) is bounded for all sufficiently small $\Delta h$: $0<\Delta h<\gamma\frac{\Delta \varepsilon \sqrt{p}}{\sqrt{2}}$.

The proof of Step 1 for A2RP LHS follows from above because $\sum_{k=1}^{\infty} \p{D^{A}_{\mathcal{L},n_{k}}(\xh_{k}) - \mathcal{G}_{\xh_k} \leq - \Delta \varepsilon} \leq \sum_{k=1}^{\infty} \p{D^{S,1}_{\mathcal{L},{n_k/2}}(\xh_k)- \mathcal{G}_{\xh_k} \leq - \Delta \varepsilon}+\sum_{k=1}^{\infty} \p{D^{S,2}_{\mathcal{L},{n_k/2}}(\xh_k)- \mathcal{G}_{\xh_k} \leq - \Delta \varepsilon}$ by the same probability inequality used in (\ref{ineq: P(A+B<C)}).

For SRP AV, there is no need to split and one could use the fact that $g_\mathcal{A}(\xh,\tilde{\xi}^i) -g_\mathcal{A}(x^*_{\mathcal{A},\min}, \tilde{\xi}^i)$ is bounded by $4C$ and independent and identically distributed for all $i=1,\ldots,n_k/2$.
So, one could directly invoke Hoeffding's inequality after applying the Chernoff bound and taking supremum over $X$. 
The proof of AV with A2RP is essentially the same because two independent sets of $n_k/4$ AV pairs form a set of $n_k/2$ AV pairs.
\medskip

\noindent{\bf \underline{Step 2.}}  
For a given $(s,a)$ pair, suppose  $\var{D^{a}_{s,n_{k}}(\xh_{k})}=0$ for some $k$. Then $D^{a}_{s,n_{k}}(\xh_{k})=\mathcal{G}_{\xh_{k}}$, w.p.1.
In this case, the probability in (\ref{ineq: Fatou}) is zero, and this term does not contribute to the summation.
Therefore, we continue with positive-variance terms below. 
Consider first $(s,a)=(\mathcal{L},S)$. 
Suppose CLT for LHS in the form of \eqref{eq:clt} holds for the remaining terms. 
Then, the left-hand side of (\ref{ineq: Fatou}) is bounded above by
\begin{small}
\begin{align}
& \sum_{k=1}^{\infty} \limsup_{\Delta h \downarrow 0} \p{D^{S}_{\mathcal{L},n_{k}}(\xh_{k}) - \mathcal{G}_{\xh_k} \leq -\Delta h \, \sqrt{\texttt{SV}^{S}_{\mathcal{L},n_{k}}(\xh_{k})} } \label{ineq: fatous1}  
\\
\leq & \mathlarger{\sum_{k=1}^{\infty}} \limsup_{\Delta h \downarrow 0}  \mathlarger{\int_{\xh_k}} \p{ \left.
\frac{D^{S}_{\mathcal{L},n_{k}}(\xh_{k}) - \mathcal{G}_{\xh_k}}{\sqrt{\var{D^{S}_{\mathcal{L},n_{k}}(\xh_{k})}}}  \leq -\Delta h \sqrt{n_k-1} \frac{\sqrt{\texttt{SV}^{S}_{\mathcal{L},n_k}(\xh_k)}}{\sigma_{\mathcal{I}}(\xh_k)} 
\right|
\xh_k}
dP_{\xh_k} \label{ineq: lvar}
\\
\leq & \mathlarger{\sum_{k=1}^{\infty}\int_{\xh_k}}\limsup_{\Delta h \downarrow 0}  \p{  \left.
\frac{D^{S}_{\mathcal{L},n_{k}}(\xh_{k}) - \mathcal{G}_{\xh_k}}{\sqrt{\var{D^{S}_{\mathcal{L},n_{k}}(\xh_{k})}}}  \leq -\sqrt{c_p+2p\ln^2k} \frac{\sqrt{n_k-1}}{\sqrt{n_k}}\frac{\sqrt{\texttt{SV}^{S}_{\mathcal{L},n_k}(\xh_k)}}{\sigma_{\mathcal{I}}(\xh_k)} \right| 
\xh_k}dP_{\xh_k} \label{ineq: lfatous2}
\\
\leq& \alpha, \notag  
\end{align}
\end{small}where  (\ref{ineq: fatous1}) and (\ref{ineq: lfatous2}) invoke Fatou's Lemma. 
Inequality (\ref{ineq: lvar}) holds because the variance of an LHS sample mean with sample size $n_k$ is at most the variance of an IID sample mean with sample size $n_k - 1$:  $\var{D^{S}_{\mathcal{L},n_{k}}(\xh_{k})} \leq \frac{\sigma^2_{\mathcal{I}}(\xh_k)}{n_k-1}$ \citep{owen1998latin}. 
And, (\ref{ineq: lfatous2}) uses the $n_k$ inequality from Table~\ref{tb: size}.
The rest of the proof follows
 by noting that 
(i) $\Delta h \downarrow 0$ ensures $n_k \to \infty$,  
(ii) $\liminf_{\Delta h \downarrow 0} \big({\sqrt{\smash[b]{\texttt{SV}^{S}_{\mathcal{L},n_k}(\xh_k)}}}/{\sigma_{\mathcal{I}}(\xh_k)}\big)\geq 1$, w.p.1 by \eqref{eq: true} and \eqref{eq:sv}, 
(iii) by using asymptotic normality through \eqref{eq:clt}, combined with a known bound $\p{Z\geq t}\leq \frac{1}{\sqrt{2\pi}}\frac{\exp (-t^2/2)}{t}$ where $Z$ follows a standard normal distribution with $t>0$, 
 and (iv) from the definition of $c_p$.

Recall that CLT for LHS given in \eqref{eq:clt} cannot be used when $f(\xh_k , \cdot)$ and $f(x^*_{\mathcal{L},\min},\cdot)$ are additive functions; remember the discussion at the end of Section~\ref{ssec:prelim}. 
In this case, the proof can be modified to multiply each side of the inequality inside the probability in \eqref{ineq: fatous1} by $\sqrt{n_k}/\sigma_{\mathcal{I}}(\xh_k)$, continuing along the same lines, and using the fact that $\sqrt{n}\left(D^{S}_{\mathcal{L},n}(\xh)-\mathcal{G}_{\xh}\right)$ converges in probability to zero as $n \rightarrow \infty$ in this case \citep{stein_87}. 
This gives the probability to be zero, and again such terms do not contribute. 

For A2RP LHS, the proof is similar to above with modifications to (\ref{ineq: lvar}) and appropriate changes to notation.
In particular, for A2RP, $\var{D^{A}_{\mathcal{L},n_{k}}(\xh_{k})}=\frac{1}{2}\var{D^{S}_{\mathcal{L},n_{k}/2}(\xh_{k})} \leq \frac{1}{2}\frac{\sigma^2_{\mathcal{I}}(\xh_k)}{n_k/2-1}$.

The proof of Step 2 using AV with SRP/A2RP follows along the same lines but considers $n_k /2$ AV pairs as IID observations---e.g., $\var{D^{S}_{\mathcal{A},n_{k}}(\xh_{k})} = \frac{\sigma^2_{\mathcal{A}}(\xh_k)}{n_k / 2}$---and is therefore omitted. 
\end{proof}
\medskip 

Before we move forward, we present another set of conditions under which the sublinear sample sizes can be used within the LHS sequential sampling procedure. 
This set of conditions does not require assumption \ref{(A2)}, but, as we shall see shortly, it is more restrictive in other aspects.  
To begin, recall the partial ANOVA decomposition of $f(\xh,\tilde{\xi})$ for a given $\xh \in X$
\begin{equation*}
f(\xh,\tilde{\xi})=\e{f(\xh,\tilde{\xi})}+f_{add}(\xh,\tilde{\xi})+f_{resid}(\xh,\tilde{\xi}), 
\end{equation*}
where $f_{add}(\xh,\tilde{\xi})=\sum^{d_{\xi}}_{j=1}f_j(\xh,\tilde{\xi}_j)$ is the additive term and $f_{resid}$ is the residual term, which includes higher-order terms \citep{stein_87}. 

In the above ANOVA decomposition,  $f_j(\xh,\xi_j)=\e{f(\xh,\tilde{\xi})\big|\tilde{\xi}_j=\xi_j}-\e{f(\xh,\tilde{\xi})}$, and we have $\e{f_j(\xh,\tilde{\xi}_j)}=\e{f_{resid}(\xh,\tilde{\xi})}=0$.  
With  $f_j$ defined as above,  $\e{f(\xh,\tilde{\xi})}+f_{add}(\xh,\tilde{\xi})$ provides the best additive fit to $f(\xh,\tilde{\xi})$ in the sense that it minimizes $\e{\big(f(\xh,\tilde{\xi})-g_{add}(\xh,\tilde{\xi})\big)^2}$ among all additive functions $g_{add}(\xh,\tilde{\xi})$.
So, $f(\xh,\tilde{\xi})$ being additive is equivalent to having a zero residual term in its ANOVA decomposition, w.p.1.  
As discussed, additive functions satisfy \ref{(A2)} under certain continuity conditions in $\tilde{\xi}$.

When $f_{resid}$ is not w.p.1 zero and $f$ is not monotone, on the other hand, we cannot guarantee \ref{(A2)}. 
However, 
by \ref{(Aso)}\ref{Aso2}, we are guaranteed to have a bounded residual term $-\infty < -m \leq f_{resid}(\xh,\tilde{\xi}) \leq M <\infty $ for all $\xh \in X$, w.p.1 for some $m, M > 0$.
Consequently, $f(\xh,\tilde{\xi})-\e{f(\xh,\tilde{\xi})}$ is bounded below and above by additive functions $f_{add}(\xh,\tilde{\xi})-m$ and $f_{add}(\xh,\tilde{\xi})+M$. 
In this case,  our main results hold  under an additional condition on the difference of parameters $\varepsilon - \varepsilon' $.
\medskip

\begin{thm}  \label{thm:cpres}
	Suppose \ref{(A1)} and \ref{(Aso)} hold. 
	Let $m, M$ be two finite constants such that at least one of $m>0$ or $M>0$, and they give the range of the residual terms in ANOVA decompositions of $f$: $-m \leq \inf_{\xh \in X, \xi \in \Xi}f_{resid}(\xh,\xi) \leq \sup_{\xh \in X, \xi \in \Xi}f_{resid}(\xh,\xi) \leq M$. 
	Suppose the univariate functions $f_j(\xh,\cdot)$ in ANOVA decompositions are continuous or the set of points they are discontinuous has Lebesgue measure zero for all $\xh \in X$. 
	Let $\varepsilon > \varepsilon' > 0$, $p > 0$, and $0 < \alpha < 1$ be fixed. Suppose $\varepsilon - \varepsilon' > m+M$.
	Consider the sequential sampling procedure where the sample sizes are increased according to Table~\ref{tb: size}.  
	If the procedure stops at iteration $T$ according to (\ref{T}), then for $(s, a)=(\mathcal{L}, S), (\mathcal{L}, A)$,
\begin{enumerate}
\item[(i)]  $P(T < \infty) = 1$,
\item[(ii)] $\liminf_{h \downarrow h'} \p{\mathcal{G}_{\xh_T} \leq h\sqrt{\texttt{SV}^{a}_{s,n_T}(\xh_T)} + \varepsilon} \geq  1- \alpha$.
\end{enumerate}
\end{thm}

\begin{proof}
We only mention differences to the proof of Theorem \ref{thm: cp}; so we only focus on part {\it (ii)}. 
Let $\Delta \varepsilon = \varepsilon - \varepsilon'$. 
For any  $\xh \in X$, $n \geq 2$, and $(s,a)$ pair, using the ANOVA decomposition of $f$, we can write $D^{a}_{s,n}(\xh)$ in a similar form:   $D^{a}_{s,n}(\xh) =\mathcal{G}_x+D^{a}_{s,n,add}(\xh)+D^{a}_{s,n,redis}(\xh)$. 
Here, $-(m+M)\leq D^{a}_{s,n,redis}(\xh)\leq m+M$ for all $\xh \in X$, w.p.1.  
Using this notation and the lower bound on $D^{a}_{s,n,redis}(\xh)$, we obtain
\begin{small}
\begin{equation}
\sum_{k=1}^{\infty} \p{D^{a}_{s,n_{k}}(\xh_{k}) - \mathcal{G}_{\xh_k} \leq - \Delta \varepsilon} \leq \sum_{k=1}^{\infty} \p{D^{a}_{s,n,add}(\xh_{k}) \leq -
\big( \Delta \varepsilon-(m+M) \big)
}.
\label{ineq: additive}
\end{equation}
\end{small}Because $D^{a}_{s,n,add}(\xh)$ is an additive function with $\e{D^{a}_{s,n,add}(\xh)}=0$ and $\Delta \varepsilon-(m+M)>0$, (\ref{ineq: additive}) is bounded above for all sufficiently small positive $\Delta h$ by Step 1 of Theorem \ref{thm: cp}'s proof.
Observe that if $\Delta \varepsilon-(m+M)\leq 0$, we cannot establish an analogous inequality (\ref{ineq: boundedLHS}), and the resulting infinite sum diverges. 
The rest of the proof is similar to that of Theorem \ref{thm: cp}. 
\end{proof}

Theorem \ref{thm:cpres} implies that if $m+M$ is sufficiently small---so $f$ is mostly additive---it may be reasonable to use the sublinear sample sizes of Table~\ref{tb: size}. 
However, when $m+M$ is large, the additional condition ${\varepsilon -\varepsilon'} > m+M$ may yield impractically large confidence intervals for many problems.  
Furthermore, the values of $m,M$ may not be known. 
Then, an alternative way discussed next can be used.

\subsection{Finite Stopping and Asymptotic Validity under Superlinear Sample Sizes}
\label{sssec:cpq}

We provide a variant of Theorem \ref{thm: cp} under less restrictive assumptions.
The price to pay for the less restrictive assumptions is that, now, the sample sizes must be increased at a higher rate. 
In particular, for a given $q>1$, the sample-size increases are updated to 
\begin{equation} \label{ineq: ss_cpq}
\begin{matrix}
\frac{n_k}{2} \geq \left(\frac{1}{h-h'}\right)^2(c_{p,q}+2pk^{q}) & \text{for $s=\mathcal{A}$,} & \text{ and } & {n_k} \geq \left(\frac{1}{h-h'}\right)^2(c_{p,q}+2pk^{q}) & \text{for $s=\mathcal{L}$,}
\end{matrix}
\end{equation}
where $p>0$ and $c_{p,q}=\max\left\lbrace 2\ln\left(\sum^{\infty}_{j=1}\exp(-pj^{q})/\sqrt{2\pi}\alpha\right),1\right\rbrace$.
For A2RP and procedures that use AV, we again assume the sample sizes obey the `even' and `multiple of four' requirements summarized in Table~\ref{tb: size}.
Unlike before, the above sample sizes are increased of order $O(k^q)$; so they are superlinear in the iteration number $k$.  
A typical value of $q=1.5$ is used in many applications.

For the less restrictive assumptions, we update the following:   
For LHS, we remove assumption \ref{(A2)}. 
For AV, we replace the boundedness assumption of \ref{(Aso)}\ref{Aso2} by the less restrictive finite second moment assumption of 
\begin{description}[1cm]
\item[(A3)\label{(A3)}] $\sup_{x \in X} E[f(x,\tilde{\xi})]^2 < \infty$.
\end{description}
We still need the boundedness condition \ref{(Aso)}\ref{Aso2} for LHS to invoke CLT for LHS.
Of course, if \ref{(Aso)}\ref{Aso2} is satisfied, \ref{(A3)} holds.
The following theorem establishes asymptotic validity of AV and LHS under the above less restrictive assumptions.
\medskip 

\begin{thm}\label{thm: cpq}
	Suppose \ref{(A1)} and \ref{(Aso)} hold, 
where for AV \ref{(Aso)}\ref{Aso2} is replaced by \ref{(A3)}.
Let $\varepsilon > \varepsilon' > 0$, $p > 0$, $q > 1$, and $0 < \alpha < 1$ be fixed.
Consider the sequential sampling procedure where the sample sizes are increased according to (\ref{ineq: ss_cpq}).
If the procedure stops at iteration $T$ according to (\ref{T}), then for all  $(s, a)=(\mathcal{A}, S), (\mathcal{A}, A), (\mathcal{L}, S), (\mathcal{L}, A)$,
\begin{enumerate}
\item[(i)]  $P(T < \infty) = 1$,
\item[(ii)] $\liminf_{h \downarrow h'} \p{\mathcal{G}_{\xh_T} \leq h\sqrt{\texttt{SV}^{a}_{s,n_T}(\xh_T)} + \varepsilon} \geq  1- \alpha$.
\end{enumerate}
\end{thm}

\begin{proof}
The proof follows the same steps of the proof of Theorem \ref{thm: cp} except Step 1 of part {\it (ii)}.
Let $(s,a)=(\mathcal{L}, S)$.
Step 1 of Theorem \ref{thm: cp} is changed starting from (\ref{ineq: P(A+B<C)}):
\begin{small}
\begin{align}
\sum_{k=1}^{\infty} \p{D^{S}_{\mathcal{L},n_{k}}(\xh_{k}) - \mathcal{G}_{\xh_k} \leq - \Delta \varepsilon} \leq&\sum_{k=1}^{\infty} \int_{\xh_k}\p{\left(D^{S}_{\mathcal{L},n_{k}}(\xh_{k}) - \mathcal{G}_{\xh_k}\right)^2 \geq  \Delta \varepsilon^2\Bigg|\xh_k}dP_{\xh_k}\nonumber\\
\leq&\sum_{k=1}^{\infty} \int_{\xh_k}\e{\left(D^{S}_{\mathcal{L},n_{k}}(\xh_{k}) - \mathcal{G}_{\xh_k}\right)^2  \Bigg|\xh_k}dP_{\xh_k} \Delta \varepsilon^{-2} \label{ineq: MI}\\
\leq&\sum_{k=1}^{\infty} \int_{\xh_k}\frac{1}{n_k-1}\var{f(\xh_k,\tilde{\xi})-f(x^*_{\mathcal{I},\min},\tilde{\xi})\Bigg|\xh_k}dP_{\xh_k} \Delta \varepsilon^{-2} \label{ineq: cpq_owen}\\
\leq&\sup_{x \in X}\left\lbrace\sigma^2_{\mathcal{I},\max}(x) \right\rbrace \Delta \varepsilon^{-2} \sum_{k=1}^{\infty}\frac{1}{n_k-1},\label{ineq: cpq_bd}
\end{align}
\end{small}where (\ref{ineq: MI}) follows from Markov's inequality, and (\ref{ineq: cpq_owen}) holds by using a conditional version of the fact  $\var{D^{S}_{\mathcal{L},n}(x)} \leq \frac{\sigma^2_{\mathcal{I}}(x)}{n-1}$ \citep{owen1998latin}.
The right-hand side of (\ref{ineq: cpq_bd}) is bounded because (i) $\sup_{x \in X}\left\lbrace\sigma^2_{\mathcal{I},\max}(x) \right\rbrace $ is bounded by \ref{(A3)}, and (ii) the infinite sum is bounded by $n_k \geq 2$ and (\ref{ineq: ss_cpq}).
The proof of LHS with A2RP follows from above, and the proofs with AV follow with minor modifications to notation and the variance term as discussed before. 
\end{proof}

The above analysis showed the validity of the resulting confidence intervals asymptotically, i.e., as $h \downarrow h'$ or equivalently as $n_k \rightarrow \infty$. 
Even in much simpler settings of forming confidence intervals on  population means, validity of sequential confidence intervals have  been shown  asymptotically \citep{chow1965asymptotic,nadas1969extension}.  But they have been successfully used in many real-world applications.  
Furthermore, due to the nature of the optimality gap estimators---which contain differences of correlated random variables, where the correlations occur because of the dependence of optimal solutions $x_{s,n}^*$ to all the sample points---it is not straightforward to establish conditions on \ref{SP} under which AV and LHS guarantee variance reduction. 
Nevertheless, these sampling methods have been empirically observed to reduce both variance and bias for many problems. 
Therefore, 
in the next section, we conduct numerical experiments to test the small-sample behavior of the procedures and the effectiveness of AV and LHS 
in the sequential setting.  
\vspace*{-0.1in}

\section{Computational Experiments and Further Analysis}
\label{sec:exp}
\vspace*{-0.05in}

\subsection{Test Problems}
\label{ssec: tp}
\vspace*{-0.05in}

We performed experiments on eight test problems from the literature.
Table \ref{tb: tp} summarizes the characteristics of these problems. 
All test problems are two-stage stochastic linear programs with recourse, and they all satisfy \ref{(Aso)}.
The first three problems in Table \ref{tb: tp}---APL1P \citep{infanger1992monte}, PGP2 \citep{higle1996stochastic} and LandS3 \citep{sen2014mitigating}---are power generation and expansion problems.
They can be solved to optimality; so, we can calculate optimality gaps and  variances exactly.
They allow us to fully analyze the results.

The next five problems are large scale, and they have no known exact optimal solutions. 
DB1 \citep{donohue1995upper, mak1999monte} and 20TERM \citep{mak1999monte}  are vehicle allocation problems. 
SSN \citep{sen1994network} is a large-scale telecommunications network optimization problem, and STORM \citep{mulvey1995new} is an aircraft scheduling problem.  
All the aforementioned problems satisfy the monotonicity property.
Therefore, we use the smaller sample-size formula of Section \ref{ssec:Asym}.
On the other hand, BAA99-20 \citep{sen2014mitigating}---a twenty-product version of a two-product inventory model of  \citet{bassok1999single}---does not satisfy the monotonicity property, and we use the larger sample-size formula of Section \ref{sssec:cpq}; see column `Case' of Table \ref{tb: tp}.

\begin{table}[tbh]
\centering
\caption{Characteristics of Test Problems}
\label{tb: tp}
{\small
\begin{tabular}{c|cccc|c}  \toprule
Problem  & Application            & \begin{tabular}[c]{@{}c@{}}1st-/2nd-stage\\ \# of variables\end{tabular} & \begin{tabular}[c]{@{}c@{}}\# of Random\\ Parameters\end{tabular} & \begin{tabular}[c]{@{}c@{}}\# of Total\\ Scenarios\end{tabular} & Case                                                                     \\ \toprule
APL1P    & power expansion        & 2/9                                                                    & 5                                                                 & 1,280                                                           & \multirow{6}{*}{\begin{tabular}[c]{@{}c@{}} \S\ref{ssec:Asym} \end{tabular}} \\ 
PGP2     & power generation       & 4/16                                                                   & 3                                                                 & 576                                                             &                                                                          \\
LandS3   & power generation       & 4/12                                                                   & 3                                                                 & $10^6$                                                          &                                                                          \\ 
DB1      & vehicle allocation     & 5/102                                                                  & 46                                                                & $4.5\times 10^{25}$                                             &                                                                          \\
SSN      & telecommunications     & 89/706                                                                 & 86                                                                & $10^{70}$                                                       &                                                                          \\
STORM    & air freight scheduling & 121/1259                                                               & 118                                                               & $5^{118}$                                                       &                                                                          \\
20TERM   & vehicle allocation     & 63/764                                                                 & 40                                                                & $1.1\times 10^{12}$                                             &                                                                          \\ \toprule
BAA99-20 & inventory model        & 20/250                                                                 & 20                                                                & $10^{34}$                                                       & \S\ref{sssec:cpq}                                                                    \\ \bottomrule
\end{tabular}}
\end{table}

The different characteristics of these test problems allow us to experiment under varying conditions. 
Both PGP2 and DB1 appear to have flat objective function values.  
They tend to produce a small number of optimal solutions to the SAA problem. 
When a candidate solution that is not optimal coincides with one of these solutions, then optimality gap estimators incorrectly estimate zero optimality gaps. 
This reduces the effectiveness of the quality assessment procedures for these two problems. 
STORM, on the other hand, produces a near-optimal solution with a relatively small number of scenarios.  
Consequently, STORM is not expected to  have large bias, and it allows us to test our results under low bias. 
In contrast, it is harder to obtain optimal solutions to SNN even with large sample sizes, allowing us to conduct tests when there may be high bias.  
As mentioned, we use BAA99-20 to test when monotonicity cannot be verified.  
\vspace*{-0.09in}

\subsection{Experimental Setup}
\label{ssec: es}
\vspace*{-0.03in}

Sequential sampling procedures need several input parameters. 
We set $\alpha=0.10$, $\varepsilon =2 \times 10^{-7}$ and $\varepsilon'= 1 \times 10^{-7}$ for all problems and procedures. 
For the smaller sample-size formula of  Section \ref{ssec:Asym}, we use $p=1.91 \times 10^{-1}$ and  $c_p=8.146$.
For the larger sample size of Section \ref{sssec:cpq}, we use $p=4.67 \times 10^{-3}$, $q=1.5$, and $c_{p,q}=9.69945$.
Next, we set $n_k$ as the smallest integer that is a multiple of four satisfying the AV formulas in Table \ref{tb: size} and (\ref{ineq: ss_cpq}).
This way, all procedures use the same number of observations $n_k$ at each iteration $k$ for a given test problem.
Hence, we can evaluate the procedures under similar conditions.

A desired initial sample size $n_1$ dictates a value of $\Delta h=h-h'$.
For PGP2, we use $n_1=100$. For the other small-scale problems, we set $n_1=200$, and for all large-scale problems, we use $n_1=500$.
Given $\Delta h$, we need to select $h'$---an important component of stopping rule ($\ref{T}$).
For a problem, we perform 25 replications for each procedure with the desired $n_1$ for 5 iterations ignoring the stopping rule ($\ref{T}$).
Then, $h'$ is estimated as $h'=0.8\times\left(\texttt{GAP}^a_{s, avg}/ \sqrt{\texttt{SV}^a_{s,avg}}\right)$, where $\texttt{GAP}^a_{s,avg}$ and ${\texttt{SV}^a_{s,avg}}$ are the average of 25 $\texttt{GAP}^a_s(\xh_5)$ and $\texttt{SV}^a_s(\xh_5)$ estimates, respectively, for a given ($s,a$) pair.
The factor $0.8$ aims to make the stopping rule ($\ref{T}$) tighter.

Table \ref{tb: para} summarizes $\Delta h$, $n_1$, and $h'$ for all problems and procedures.
Generally, $h'$ of AV is the largest among $s=\mathcal{I}, \mathcal{A}, \mathcal{L}$ for a fixed $a=S, A$.
This is because AV tends to have a smaller sample variance.
We will discuss this issue further soon. 
Similarly, $h'$ of A2RP estimators are larger than their SRP counterparts for each $s=\mathcal{I}, \mathcal{A}, \mathcal{L}$.
This is because A2RP \texttt{GAP} estimators tend to be more conservative.

\begin{table}[tbh]
	\centering
	\caption{Estimated $h'$ for a Given $\Delta h$} 
	\label{tb: para}
	\begin{threeparttable}
	{\small 
		\begin{tabular}{c|cc|rrrrrr} \toprule
			& \multirow{2}{*}{$\Delta h$}           &     \multirow{2}{*}{$n_1$}        & \multicolumn{6}{c}{$h'$ for ($s,a$)\tnote{1}\ =\ }  \\ 
			&  &      & $(\mathcal{I},S)$ & $(\mathcal{A},S)$ & $(\mathcal{L},S)$ & $(\mathcal{I},A)$ & $(\mathcal{A},A)$ & $(\mathcal{L},A)$ \\ \toprule
			APL1P    & 0.2855   & 200 & 0.036             & 0.053             & 0.020             & 0.049             & 0.085             & 0.029             \\
			PGP2     & 0.4039   & 100  & 0.073             & 0.105             & 0.084             & 0.135             & 0.204             & 0.129             \\
			LandS3   & 0.2855   & 200 & 0.047             & 0.064             & 0.033             & 0.067             & 0.076             & 0.058             \\
			DB1      & 0.1806   & 500 & 0.025             & 0.042             & 0.027             & 0.046             & 0.038             & 0.027             \\
			SSN      & 0.1806   & 500 & 0.087             & 0.118             & 0.057             & 0.125             & 0.175             & 0.112             \\
			STORM    & 0.1806   & 500 & 0.032             & 0.042             & 0.012             & 0.048             & 0.053             & 0.039             \\
			20TERM   & 0.1806   & 500 & 0.025             & 0.040             & 0.022             & 0.034             & 0.057             & 0.031             \\
			BAA99-20 & 0.1971   & 500 & 0.077	&	0.118	&	0.070	&	0.104	&	0.156	&	0.090            \\ \bottomrule
		\end{tabular}
		}
		\begin{tablenotes}
			\item[1] \footnotesize{$h'$ of $(2\mathcal{I},a)$ is set to $\sqrt{2}h'$ of  $(\mathcal{I},a)$ for $a=S,A$; e.g., for APL1P,
			$h'$ of $(2\mathcal{I},S) = \sqrt{2}\times0.036 = 0.051$ 
			}
		\end{tablenotes}
	\end{threeparttable}
\end{table}

We performed 300 replications for the small-scale problems and 100 replications for the large-scale problems.
To further evaluate the procedures under similar conditions, each independent replication used the same stream of (i) pseudo-random numbers and (ii) candidate solutions $\xh_k$,  $k=1,2,\ldots$ across the procedures.
The candidate solutions $\xh_k$ were generated independently by solving SAAs with $n_k$ IID observations.  At each iteration $k$, the solution-generating SAAs used new $n_{k}$ IID observations. 
Finally, because LHS must generate a new LHS sample every time the sample size is increased, all variants generate new samples when assessing solution quality.  

All experiments were performed on the Oakley cluster of \citet{OhioSupercomputerCenter1987}, which consists of 8,328 total cores, Intel Xeon $\times$5650 CPUs and  4.0 gigabytes memory per core.
The problems were solved in a non-parallel fashion with the regularized decomposition algorithm of \citet{ruszczynski1986regularized} and \citet{ruszczynski1997Accelerating} written in C\texttt{++}.
We used Mersenne Twister \citep{Matsumoto1998Mersenne} random number generator for all experiments.
\vspace*{-0.09in}

\subsection{Results and Analysis}
\label{sssec: Comparison}

Before we present the results of our computational experiments, let us first examine the {variances} $\sigma^2_{s}(\xh)$ defined in \eqref{eq: true} for the three sampling methods $s=\mathcal{I}, \mathcal{A}, \mathcal{L}$ in more detail.
Let $\xh \in X$ be given, and suppose all variances are finite. 
If \ref{SP} has a unique optimal solution, i.e., $X^*$ is a singleton set, then $x^*_{\mathcal{A},\min}$ and  $x^*_{\mathcal{I},\min}$, defined right after equation \eqref{eq: true}, are both equal to the optimal solution $x^*$ and  
\begin{align}
	\sigma^2_{\mathcal{A}}(\xh) & = \var{\frac{1}{2}\left[f(\xh,\tilde{\xi}_{\mathcal{A}})
- f(x^*_{\mathcal{A},\min}, \tilde{\xi}_{\mathcal{A}})
+f(\xh,\tilde{\xi}_{\mathcal{A}'}) - f(x^*_{\mathcal{A},\min}, \tilde{\xi}_{\mathcal{A}'})\right]} \nonumber
	\\
	& = \frac{1}{2}\var{f(\xh,\tilde{\xi}) - f(x^*, \tilde{\xi})}
+ \frac{1}{2}\cov{f(\xh,\tilde{\xi}_{\mathcal{A}}) - f(x^*, \tilde{\xi}_{\mathcal{A}}), f(\xh,\tilde{\xi}_{\mathcal{A}'})- f(x^*, \tilde{\xi}_{\mathcal{A}'})} \label{eq:longavvar}
	\\
	& \leq \frac{1}{2}\sigma^2_{\mathcal{I}}(\xh)+\frac{1}{2}\sigma^2_{\mathcal{I}}(\xh).  \label{ineq: cauchy1}
\end{align}
The second term on the right-hand side of inequality (\ref{ineq: cauchy1}) follows from Cauchy-Schwarz inequality. 
So, when there is a unique optimal solution, we expect the variance of AV to be at most as large as those of IID and LHS:  $\sigma^2_{\mathcal{A}}(\xh)\leq \sigma^2_{\mathcal{I}}(\xh)=\sigma^2_{\mathcal{L}}(\xh)$.  
On the other hand, if \ref{SP} has multiple optimal solutions,
by a similar reasoning we obtain $\sigma^2_{\mathcal{A}}(\xh)\leq \sigma^2_{\mathcal{I},\max}(\xh)=\sigma^2_{\mathcal{L},\max}(\xh)$, where $\sigma^2_{s,\max}(\xh)$ for $s=\mathcal{I}, \mathcal{L}$ are defined in \eqref{eq:sigma_max}.

Unlike widely known properties of AV, the above results always hold regardless of the preservation of negative correlation through conditions like monotonicity. 
They hold even if the covariance is increased after applying $f(\xh,\cdot)-f(x^*_{\mathcal{A},\min}, \cdot)$ to the antithetic pair. 
This is because the above analysis compares the variance of an averaged {\it pair} of antithetic observations to the variance of an {\it individual} observation. 
In contrast, classical AV variances are compared on sample means formed via AV or IID using the {\it same number} of observations.

Based on the above discussion, to have a fair comparison in our numerical results, we examined the performance of AV relative to the average of two IID pairs. 
We denote this sampling method as $s=2\mathcal{I}$. 
For $2\mathcal{I}$, we used the same estimators defined for AV in Section \ref{ssec:av} but with two IID observations instead of an AV pair. 
That is, instead of 
$g_{\mathcal{A}}(x,\tilde{\xi})=\frac{1}{2}\big(f(x,\tilde{\xi}_{\mathcal{A}})+f(x,\tilde{\xi}_{\mathcal{A}'})\big)$, where  $(\tilde{\xi}_{\mathcal{A}},\tilde{\xi}_{\mathcal{A}'})$ is an AV pair, 
we used 
$g_{2\mathcal{I}}(x,\tilde{\xi})=\frac{1}{2}\big(f(x,\tilde{\xi}_{\mathcal{I}})+f(x,\tilde{\xi}_{\mathcal{I}'})\big)$, where  $(\tilde{\xi}_{\mathcal{I}},\tilde{\xi}_{\mathcal{I}'})$ are two IID observations. 
All other notation is updated accordingly. 
Theorems \ref{thm: cp} and \ref{thm: cpq} hold for this sampling scheme by the regular IID sequential theory with minor updates in notation.

Suppose for simplicity that \ref{SP} has a unique optimal solution $x^*$. 
Observe that $\sigma^2_{2\mathcal{I}}(\xh) = \var{g_{2\mathcal{I}}(\xh,\tilde{\xi})-g_{2\mathcal{I}}(x^*,\tilde{\xi})} =\frac{1}{2}\sigma^2_{\mathcal{I}}(\xh)$.  So, the variance of AV is reduced with respect to $2\mathcal{I}$ only if the covariance term in \eqref{eq:longavvar} is negative.
To illustrate these results numerically, we picked two candidate solutions $\xh\in X$ from APL1P and PGP2 and calculated their variances.
Both problems have unique optimal solutions and are small-scale; so we can calculate these quantities exactly.

\begin{table}[!htb]
\centering
\caption{True Standard Deviations for Given $\xh$}
\label{tb: trvar}
{\small
\begin{tabular}{c|crrr}
    \toprule
      & $\xh$           & $\sigma_{\mathcal{I}}(\xh)=\sigma_{\mathcal{L}}(\xh)$ &  {$\sigma_{2\mathcal{I}}(\xh)=\frac{\sigma_{\mathcal{I}}(\xh)}{\sqrt{2}}$} &    $\sigma_{\mathcal{A}}(\xh)$    \\  \hline
APL1P & (1111.11, 2300.00)  & 1,893.03                                            &  {1,338.57}         &   860.05                    \\
PGP2  & (1.5, 5.5, 5, 4.5)  & 82.69                                               & {58.47}        &       58.25                    \\
   \bottomrule
\end{tabular}
}
\end{table}

Table \ref{tb: trvar} lists the specific candidate solutions $x$ for APL1P and PGP2, along with their corresponding standard deviations $\sigma_s(\xh)$ for each sampling method $s=\mathcal{I, L,}\ 2\mathcal{I}, \mathcal{A}$. 
Table \ref{tb: trvar} shows that, for these candidate solutions, the AV standard deviation is indeed quite smaller than its IID or LHS counterparts.
However, compared to $2\mathcal{I}$, AV significantly reduces the standard deviation for APL1P, whereas it remains approximately the same for PGP2.

We end this discussion with a remark that so far we focused on `asymptotic' variances; see \eqref{eq:sv}. 
Once again assume for simplicity of discussion that \ref{SP} has a unique optimal solution. 
Then,  $\texttt{SV}^{a}_{s,n}(\xh) \rightarrow \sigma^2_{s}(\xh)$, w.p.1 as $n \rightarrow \infty$ for all $(s,a)$ pairs. 
This is similar to sample means formed by IID, 2$\mathcal{I}$, AV, and LHS {\it all} converging to the true population mean.  
In the context of $\texttt{SV}^{a}_{s,n}(\xh)$, though, we find that IID and LHS converge to the same value, whereas AV may converge to a smaller value relative to 2$\mathcal{I}$ depending on the problem and the candidate solution for that problem.  
Note that the variance and bias of the estimator $\texttt{SV}^{a}_{s,n}(\xh)$ might differ by sampling method; we will examine this in the next section.

We are now ready to present the computational results. Before we do so, let us briefly mention the experimental setup for the sampling method $2\mathcal{I}$. We  ran the sequential procedures with $2\mathcal{I}$ using the same parameter settings discussed in the previous section. The only  exception is that we adjusted the $h'$ values for $2\mathcal{I}$ by simply scaling the $h'$ values obtained for IID with $\sqrt{2}$. 
That is, $h'$ of $(2\mathcal{I},a)$  is set to $\sqrt{2}\times h'$ of $(\mathcal{I},a)$ for both solution quality assessment methods $a=S,A$.
As an example, $h'$ of $(2\mathcal{I},S)$ for APL1P is set to $\sqrt{2}\times0.036 = 0.051$; see Table~\ref{tb: para}. 
This is because, given the same observations and using the same optimal solutions to the relevant SAA problems, the $\texttt{GAP}$ estimators for $2\mathcal{I}$ and IID will be the same, whereas their standard deviations $\sqrt{\texttt{SV}}$ will differ approximately by $\sqrt{2}$.

Table \ref{tb: cp} presents the results of the experiments. 
The left portion of Table \ref{tb: cp} compares LHS with IID, and the right portion compares AV with $2\mathcal{I}$.  In each portion, column `$T$' lists the number of iterations to stop, `$h\sqrt{\smash[b]{\texttt{SV}^{a}_{s, n_T}(\xh_T)}}+\varepsilon$' tabulates the confidence interval width on $\mathcal{G}_{\xh_T}$,  and `Coverage' presents the empirical coverage probability of this interval. 
For each of these {three} quantities, average values along with 90\% half-widths are provided. 
Finally, to facilitate the comparison of alternative sampling methods LHS and AV, column `CI Ratio' lists the ratio of the average confidence interval widths of IID to LHS (left portion) and of $2\mathcal{I}$ to AV (right portion).

To find the empirical coverage probabilities listed in Table \ref{tb: cp}, we need the optimality gaps of $x_T$.   
For small-scale problems, we calculated $\mathcal{G}_{\xh_T}$ exactly. 
For large-scale problems, we solved an SAA with 50,000 LHS observations and used its optimal value as an estimate of the true optimal value. 
Then, for each solution $\xh_T$, we estimated $\mathcal{G}_{\xh_T}$ by using the same 50,000 LHS observations. 
These  $\mathcal{G}_{\xh_T}$  estimates are used to find the empirical coverage probabilities. 
To calculate the half-widths for the coverage probabilities, we used the binomial proportion confidence intervals.

\begin{table}[!htb]
\setlength{\tabcolsep}{2pt}
\centering
\caption{Results of Computational Experiments}
\label{tb: cp}
\resizebox{\textwidth}{!}{ 
\begin{tabular}{c|crrcrr|crrcr}
\cmidrule{1-7}\cmidrule{8-12}    Problem & ($s,a$) & \multicolumn{1}{c}{$T$} & \multicolumn{1}{c}{{\small $h\sqrt{\texttt{SV}^{a}_{s, n_T}(\xh_T)}+\varepsilon$}} & \multicolumn{1}{c}{\begin{tabular}{c}
  \texttt{CI}\\
  Ratio
\end{tabular}} & \multicolumn{1}{c}{Coverage} &  \  \ \ \ \     & ($s,a$) & \multicolumn{1}{c}{$T$} & \multicolumn{1}{c}{{\small $h\sqrt{\texttt{SV}^{a}_{s, n_T}(\xh_T)}+\varepsilon$}} & \multicolumn{1}{c}{\begin{tabular}{c}
  \texttt{CI}\\
  Ratio
\end{tabular}} & \multicolumn{1}{c}{Coverage} \\
\cmidrule{1-7}\cmidrule{8-12}    APL1P & ($\mathcal{I}$, $S$) & 2.55$\pm$0.20 & 125.55$\pm$10.42 & \multicolumn{1}{c}{-} & 0.92$\pm$0.03 &       & ($2\mathcal{I}$, $S$) & 2.57$\pm$0.20 & 92.65$\pm$7.77 & \multicolumn{1}{c}{-} & 0.89$\pm$0.03 \\
          & ($\mathcal{L}$, $S$) & 6.42$\pm$0.64 & 54.26$\pm$4.88 & 2.31  & 0.94$\pm$0.02 &       & ($\mathcal{A}$, $S$) & 4.06$\pm$0.39 & 37.09$\pm$3.12 & 2.50  & 0.88$\pm$0.03 \\
\cmidrule{2-6}\cmidrule{8-12}          & ($\mathcal{I}$, $A$) & 2.43$\pm$0.17 & 204.85$\pm$11.76 & \multicolumn{1}{c}{-} & 0.99$\pm$0.01 &       & ($2\mathcal{I}$, $A$) & 2.47$\pm$0.19 & 154.56$\pm$9.04 & \multicolumn{1}{c}{-} & 0.98$\pm$0.01 \\
          & ($\mathcal{L}$, $A$) & 5.58$\pm$0.57 & 108.96$\pm$7.45 & 1.88  & 1.00$\pm$0.00 &       & ($\mathcal{A}$, $A$) & 3.19$\pm$0.28 & 71.52$\pm$4.25 & 2.16  & 0.98$\pm$0.01 \\
\cmidrule{1-7}\cmidrule{8-12}    PGP2  & ($\mathcal{I}$, $S$) & 2.49$\pm$0.27 & 22.04$\pm$2.80 & \multicolumn{1}{c}{-} & 0.74$\pm$0.04 &       & ($2\mathcal{I}$, $S$) & 2.49$\pm$0.27 & 16.74$\pm$2.11 & \multicolumn{1}{c}{-} & 0.73$\pm$0.04 \\
          & ($\mathcal{L}$, $S$) & 2.29$\pm$0.26 & 25.30$\pm$3.28 & 0.87  & 0.76$\pm$0.04 &       & ($\mathcal{A}$, $S$) & 2.57$\pm$0.34 & 17.64$\pm$2.33 & 0.95  & 0.76$\pm$0.04 \\
\cmidrule{2-6}\cmidrule{8-12}          & ($\mathcal{I}$, $A$) & 3.92$\pm$0.43 & 47.94$\pm$3.68 & \multicolumn{1}{c}{-} & 0.84$\pm$0.03 &       & ($2\mathcal{I}$, $A$) & 3.96$\pm$0.44 & 37.91$\pm$2.91 & \multicolumn{1}{c}{-} & 0.84$\pm$0.04 \\
          & ($\mathcal{L}$, $A$) & 3.65$\pm$0.44 & 42.49$\pm$3.81 & 1.13  & 0.84$\pm$0.04 &       & ($\mathcal{A}$, $A$) & 3.46$\pm$0.41 & 36.26$\pm$3.01 & 1.05  & 0.86$\pm$0.03 \\
\cmidrule{1-7}\cmidrule{8-12}    LandS3 & ($\mathcal{I}$, $S$) & 4.05$\pm$0.32 & 0.13$\pm$0.01 & \multicolumn{1}{c}{-} & 0.98$\pm$0.01 &       & ($2\mathcal{I}$, $S$) & 4.08$\pm$0.33 & 0.10$\pm$0.00 & \multicolumn{1}{c}{-} & 0.96$\pm$0.02 \\
          & ($\mathcal{L}$, $S$) & 6.81$\pm$0.62 & 0.10$\pm$0.00 & 1.31  & 1.00$\pm$0.00 &       & ($\mathcal{A}$, $S$) & 10.41$\pm$0.84 & 0.05$\pm$0.00 & 1.87  & 0.99$\pm$0.01 \\
\cmidrule{2-6}\cmidrule{8-12}          & ($\mathcal{I}$, $A$) & 3.34$\pm$0.25 & 0.23$\pm$0.01 & \multicolumn{1}{c}{-} & 0.99$\pm$0.01 &       & ($2\mathcal{I}$, $A$) & 3.41$\pm$0.26 & 0.18$\pm$0.01 & \multicolumn{1}{c}{-} & 0.99$\pm$0.01 \\
          & ($\mathcal{L}$, $A$) & 3.79$\pm$0.33 & 0.19$\pm$0.01 & 1.22  & 1.00$\pm$0.00 &       & ($\mathcal{A}$, $A$) & 12.19$\pm$0.93 & 0.08$\pm$0.00 & 2.18  & 1.00$\pm$0.00 \\
\cmidrule{1-7}\cmidrule{8-12}    DB1   & ($\mathcal{I}$, $S$) & 1.60$\pm$0.16 & 6.54$\pm$1.32 & \multicolumn{1}{c}{-} & 0.82$\pm$0.06 &       & ($2\mathcal{I}$, $S$) & 1.60$\pm$0.17 & 4.85$\pm$0.98 & \multicolumn{1}{c}{-} & 0.82$\pm$0.06 \\
          & ($\mathcal{L}$, $S$) & 1.73$\pm$0.20 & 5.97$\pm$1.32 & 1.09  & 0.75$\pm$0.07 &       & ($\mathcal{A}$, $S$) & 1.42$\pm$0.12 & 5.51$\pm$0.96 & 0.88  & 0.78$\pm$0.07 \\
\cmidrule{2-6}\cmidrule{8-12}          & ($\mathcal{I}$, $A$) & 1.62$\pm$0.16 & 11.52$\pm$1.12 & \multicolumn{1}{c}{-} & 0.89$\pm$0.05 &       & ($2\mathcal{I}$, $A$) & 1.64$\pm$0.16 & 8.88$\pm$0.87 & \multicolumn{1}{c}{-} & 0.84$\pm$0.06 \\
          & ($\mathcal{L}$, $A$) & 2.40$\pm$0.43 & 9.62$\pm$1.02 & 1.20   & 0.94$\pm$0.04 &       & ($\mathcal{A}$, $A$) & 2.27$\pm$0.30 & 5.80$\pm$0.77 & 1.53  & 0.90$\pm$0.05 \\
\cmidrule{1-7}\cmidrule{8-12}    SSN   & ($\mathcal{I}$, $S$) & 6.27$\pm$0.71 & 3.28$\pm$0.13 & \multicolumn{1}{c}{-} & 1.00$\pm$0.00 &       & ($2\mathcal{I}$, $S$) & 6.23$\pm$0.74 & 2.67$\pm$0.11 & \multicolumn{1}{c}{-} & 1.00$\pm$0.00 \\
          & ($\mathcal{L}$, $S$) & 6.41$\pm$0.80 & 2.51$\pm$0.08 & 1.31  & 1.00$\pm$0.00 &       & ($\mathcal{A}$, $S$) & 8.09$\pm$1.02 & 2.52$\pm$0.10 & 1.06  & 1.00$\pm$0.00 \\
\cmidrule{2-6}\cmidrule{8-12}          & ($\mathcal{I}$, $A$) & 5.94$\pm$0.62 & 3.91$\pm$0.12 & \multicolumn{1}{c}{-} & 1.00$\pm$0.00 &       & ($2\mathcal{I}$, $A$) & 6.24$\pm$0.60 & 3.23$\pm$0.10 & \multicolumn{1}{c}{-} & 1.00$\pm$0.00 \\
          & ($\mathcal{L}$, $A$) & 1.78$\pm$0.35 & 3.22$\pm$0.10 & 1.21  & 1.00$\pm$0.00 &       & ($\mathcal{A}$, $A$) & 6.35$\pm$0.64 & 3.16$\pm$0.09 & 1.02  & 1.00$\pm$0.00 \\
\cmidrule{1-7}\cmidrule{8-12}    STORM & ($\mathcal{I}$, $S$) & 2.25$\pm$0.29 & 505.28$\pm$62.15 & \multicolumn{1}{c}{-} & 0.91$\pm$0.05 &       & ($2\mathcal{I}$, $S$) & 2.32$\pm$0.31 & 386.10$\pm$46.73 & \multicolumn{1}{c}{-} & 0.89$\pm$0.05 \\
          & ($\mathcal{L}$, $S$) & 6.54$\pm$0.95 & 400.64$\pm$66.12 & 1.26  & 0.94$\pm$0.04 &       & ($\mathcal{A}$, $S$) & 3.67$\pm$0.48 & 184.43$\pm$29.91 & 2.09  & 0.88$\pm$0.05 \\
\cmidrule{2-6}\cmidrule{8-12}          & ($\mathcal{I}$, $A$) & 2.69$\pm$0.40 & 643.71$\pm$43.14 & \multicolumn{1}{c}{-} & 0.99$\pm$0.02 &       & ($2\mathcal{I}$, $A$) & 2.59$\pm$0.35 & 491.73$\pm$32.90 & \multicolumn{1}{c}{-} & 0.98$\pm$0.02 \\
          & ($\mathcal{L}$, $A$) & 1.92$\pm$0.27 & 576.93$\pm$48.47 & 1.12  & 0.99$\pm$0.02 &       & ($\mathcal{A}$, $A$) & 4.86$\pm$0.69 & 246.99$\pm$22.35 & 1.99  & 0.98$\pm$0.02 \\
\cmidrule{1-7}\cmidrule{8-12}    20TERM & ($\mathcal{I}$, $S$) & 2.79$\pm$0.37 & 65.80$\pm$5.67 & \multicolumn{1}{c}{-} & 1.00$\pm$0.00 &       & ($2\mathcal{I}$, $S$) & 2.76$\pm$0.36 & 48.98$\pm$4.23 & \multicolumn{1}{c}{-} & 0.97$\pm$0.03 \\
          & ($\mathcal{L}$, $S$) & 3.96$\pm$0.55 & 54.33$\pm$4.31 & 1.21  & 1.00$\pm$0.00 &       & ($\mathcal{A}$, $S$) & 3.18$\pm$0.38 & 38.09$\pm$2.87 & 1.29  & 1.00$\pm$0.00 \\
\cmidrule{2-6}\cmidrule{8-12}          & ($\mathcal{I}$, $A$) & 2.85$\pm$0.34 & 81.38$\pm$5.50 & \multicolumn{1}{c}{-} & 1.00$\pm$0.00 &       & ($2\mathcal{I}$, $A$) & 2.81$\pm$0.34 & 62.98$\pm$4.28 & \multicolumn{1}{c}{-} & 1.00$\pm$0.00 \\
          & ($\mathcal{L}$, $A$) & 2.55$\pm$0.33 & 76.10$\pm$5.29 & 1.07  & 1.00$\pm$0.00 &       & ($\mathcal{A}$, $A$) & 2.53$\pm$0.35 & 50.29$\pm$3.67 & 1.25  & 1.00$\pm$0.00 \\
\cmidrule{1-7}\cmidrule{8-12}    BAA99-20 & ($\mathcal{I}$, $S$) & 11.65$\pm$1.48 & 1876.71$\pm$103.87 & \multicolumn{1}{c}{-} & 1.00$\pm$0.00 &       & ($2\mathcal{I}$, $S$) & 10.56$\pm$1.37 & 1452.68$\pm$74.15 & \multicolumn{1}{c}{-} & 1.00$\pm$0.00 \\
          & ($\mathcal{L}$, $S$) & 16.52$\pm$2.30 & 1589.76$\pm$68.49 & 1.18  & 1.00$\pm$0.00 &       & ($\mathcal{A}$, $S$) & 14.82$\pm$1.95 & 1203.62$\pm$53.54 & 1.21  & 1.00$\pm$0.00 \\
\cmidrule{2-6}\cmidrule{8-12}          & ($\mathcal{I}$, $A$) & 9.73$\pm$1.18 & 2919.57$\pm$127.62 & \multicolumn{1}{c}{-} & 1.00$\pm$0.00 &       & ($2\mathcal{I}$, $A$) & 10.37$\pm$1.13 & 2323.75$\pm$93.62 & \multicolumn{1}{c}{-} & 1.00$\pm$0.00 \\
          & ($\mathcal{L}$, $A$) & 23.97$\pm$2.51 & 2466.95$\pm$98.82 & 1.18  & 1.00$\pm$0.00 &       & ($\mathcal{A}$, $A$) & 16.93$\pm$1.98 & 1758.87$\pm$55.31 & 1.32  & 1.00$\pm$0.00 \\
\cmidrule{1-7}\cmidrule{8-12}    \end{tabular}
    }
\end{table}

Table \ref{tb: cp} reveals {the following} observations.  
First, both AV and LHS reduce confidence interval widths compared to IID. 
Second, AV typically has the higher \texttt{CI} ratios---hence the lower relative confidence interval widths---when either SRP or A2RP is fixed.
We will examine this in more detail in the next section.
The large \texttt{CI} ratio (or relatively tight confidence interval width) of AV comes with a caveat.
When SRP is used, the coverage probability  can drop below $2\mathcal{I}$; see, e.g., APL1P, DB1, and STORM. 
The low coverage probability can typically be alleviated by using A2RP.

Table~\ref{tb: cp} also shows that LHS and AV sequential procedures can take a longer time to stop relative to IID and $2\mathcal{I}$ when SRP is used.  With A2RP, when we compare the stopping iteration $T$ of LHS and AV sequential procedures with each other, we observe that (i) when LHS stops later, it typically takes slightly longer time to stop than AV, however (ii) when AV stops later than LHS, it can do so by a large margin.  This can be seen in the test problems LandS3, SSN, and STORM.

Let us now briefly discuss the performance of SRP versus A2RP in the sequential setting. 
SRP confidence interval widths are typically tighter than those of A2RP.
However, for some problems like PGP2 and DB1, SRP's smaller confidence interval width causes a less-than-desired coverage probability.
This is because, it can happen that, although $\texttt{SV}^{a}_{s, n_T}(\xh_T)$ is zero (hence \texttt{GAP}$^a_{s, n_T}$($\xh_T$) is zero), the optimality gap $\mathcal{G}_{\xh_T}$ is not.
In other words, the procedure sees $\xh_T$ as an optimal solution even though it is not. 

A2RP considerably lowers these false identifications compared to SRP because now two sample variances must be zero.
Consequently, both the confidence interval width and the coverage probability of A2RP are higher than SRP's. 
These observations---with or without AV and LHS---remain consistent with the non-sequential setting. 
In our computational experiments, within A2RP, we observed that the LHS sequential procedure has fewer false identifications compared to AV.

\subsection{Comparison of Non-Sequential and Sequential Approaches}
\label{ssec:comparison}
\vspace*{-0.03in}

We now compare the non-sequential and sequential approaches to solution quality assessment in more detail. 
Because A2RP has better small-sample behavior, we continue the analysis here focusing on A2RP.

The sequential approach is different from the non-sequential approach in a number of respects.
A major difference is their purpose:
The sequential approach aims to detect a high-quality solution, whereas the non-sequential approach just assesses the quality of a given solution.
Although their purposes are different, both approaches share a final output, a confidence interval on the optimality gap. 
However, these confidence intervals are formed differently.

Table \ref{tb: nn_ss_CI} summarizes the non-sequential and the sequential confidence intervals with A2RP for a given $\alpha$, where $z_{\alpha}$ is $1-\alpha$ quantile of the standard normal distribution.
The non-sequential confidence interval width is sum of a \texttt{GAP} estimator plus its sampling error, whereas the sequential interval width is mainly based on $\sqrt{\texttt{SV}}$. 
While both intervals ensure an asymptotic $(1-\alpha)\%$ confidence, the `$\alpha$' in the sequential procedure is hidden in the sample-size formulas---see, e.g., \eqref{nk_IID} and \eqref{ineq: ss_cpq}---and the `\texttt{GAP}' in the sequential procedure appears in the stopping rule (\ref{T}).

	\begin{table}[htb]
	\centering
	\caption{Non-Sequential and Sequential Approximate $(1-\alpha)\%$ Confidence Intervals (\texttt{CI}) with A2RP}
	\label{tb: nn_ss_CI}
	\begin{tabular}{ccc}
		\toprule
		Procedure & Non-Sequential \texttt{CI}& Sequential \texttt{CI} \\ \hline 
		&& \vspace{-.4cm}\\ 
		AV & $\left[0,\texttt{GAP}^{A}_{\mathcal{A}, n}(\xh)+z_{\alpha}\sqrt{\frac{\texttt{SV}^{A}_{\mathcal{A}, n}(\xh)}{n/2}} \right]$&$\left[0,h\sqrt{\texttt{SV}^{A}_{\mathcal{A}, n_T}(\xh_T)} +\varepsilon\right]$ \vspace{0.2cm} \\ 
		LHS & $\left[0,\texttt{GAP}^{A}_{\mathcal{L}, n}(\xh)+z_{\alpha}\sqrt{\frac{\texttt{SV}^{A}_{\mathcal{L}, n}(\xh)}{n}} \right]$&$\left[0,h\sqrt{\texttt{SV}^{A}_{\mathcal{L}, n_T}(\xh_T)} +\varepsilon\right]$ \\ 		\bottomrule
	\end{tabular}
\end{table}

For brevity, we illustrate the differences between non-sequential and sequential approaches in detail on one large-scale problem, BAA99-20.  
We then present a summary of results for all test problems. 
For each test problem, we picked a specific $\xh \in {X}$ by solving an independent SAA problem. 
For large-scale problems, we used 50,000 LHS samples to estimate $\mathcal{G}_{\xh}$ and $\sigma_{\mathcal{L}}(\xh)=\sigma_{\mathcal{I}}(\xh)$.
We simply estimated $\sigma_{2\mathcal{I}}(\xh)$ by dividing $\sigma_{\mathcal{I}}(\xh)$ by $\sqrt{2}$. 
Similarly, we used 25,000 AV pairs to estimate $\sigma_{\mathcal{A}}(\xh)$.
For small-scale problems, we directly calculated these quantities, except for $\sigma_{\mathcal{A}}(\xh)$ of LandS3, for which we also used the above estimation approach. 
To have a more direct comparison, we forced the sequence of candidate solutions $\left\lbrace\xh_1,\xh_2,\ldots \right\rbrace$ for the sequential procedure to be the same $\xh$ at every iteration.
This way, we only generate confidence intervals on the optimality gap of $\xh$ and evaluate the effect of the stopping rule. 
We performed 300 independent replications for each test problem and set the sample size of the non-sequential interval and the initial sample size of the sequential procedure as the $n_1$ values listed in Table~\ref{tb: para}. 
For sequential procedures, we used the same parameter settings discussed before.

For BAA99-20, Table \ref{tb: BAA99_result} lists (i) the average values of \texttt{GAP}, $\sqrt{\texttt{SV}}$, and \texttt{CI} estimators, along with their (ii) percentage reductions in average values,  variance and bias. 
As before, we compare LHS to IID and AV to $2\mathcal{I}$. 
Bias reductions are only presented for \texttt{GAP} and $\sqrt{\texttt{SV}}$ estimators.  
For the $\sqrt{\texttt{SV}}$ estimator, we list percentage reductions in {\it relative bias} instead of bias because true variances are different for different sampling methods. 
We calculate the relative bias for the $\sqrt{\texttt{SV}}$ estimator by dividing its bias by its corresponding estimated values of $\sigma_{s}(\xh)$, $s=\mathcal{I}, 2\mathcal{I}, \mathcal{A}, \mathcal{L}$.
For the specific candidate solution $x \in \mathbb{R}^{20}$ of BAA99-20, we estimated $\mathcal{G}_{\xh}=95.26$, $\sigma_{\mathcal{L}}(\xh)=\sigma_{\mathcal{I}}(\xh)=1,928.77$, $\sigma_{2\mathcal{I}}(\xh)=1,928.77/\sqrt{2}=1,363.85$, and $\sigma_{\mathcal{A}}(\xh)=1,269.23$ using the aforementioned methodology.

\begin{table}[htb]
	\centering
	\small
	\caption{Results of BAA99-20 with Non-Sequential (NonS) and Sequential (Seq) Approaches}
	\label{tb: BAA99_result}
	\begin{threeparttable}
	\begin{tabular}{c|c|cccc|cccc|ccc}
		\toprule
		\multirow{3}{*}{}               & \multirow{3}{*}{$s$} & \multicolumn{4}{c|}{$\texttt{GAP}^{A}_{s, n_T}(\xh)$}    & \multicolumn{4}{c|}{$\sqrt{\texttt{SV}^{A}_{s, n_T}(\xh)}$} & \multicolumn{3}{c}{\texttt{CI}}                       \\
		\cline{3-13}
		&                           & \multirow{2}{*}{average} & \multicolumn{3}{c|}{\% reduction of} & \multirow{2}{*}{average}  & \multicolumn{3}{c|}{\% reduction of} & \multirow{2}{*}{average} & \multicolumn{2}{c}{\% reduction of}\\ \cline{4-6} \cline{8-10}  \cline{12-13}
		&                           &                           & ave & var  & bias      &                            & ave & var   &  rel.bias       &                           & ave & var        \\
		\hline
		    \multirow{4}[4]{*}{NonS\tnote{1}\ } & $\mathcal{I}$     &                 919.68  &  -   & -  &  -    &    8,128.82  &  -  & -  &  -    &        1,385.56  &  -  & - \\
          & $\mathcal{L}$     &                 670.71  & 27\% &  61\%  & 30\%  &    6,445.30 & 21\%  & 56\%  & 27\%  &        1,040.11 & 25\%  & 62\% \\
\cmidrule{2-13}          & $2\mathcal{I}$    &                 919.68  &  -    &  -    & - &             5,733.58     &  -    &  -    &  - &       1,384.40  &  - & -  \\
          & $\mathcal{A}$     &                 807.92  & 12\%  & 45\%  & 14\%  &    4,596.79  & 20\% &  67\%  & 18\%  &        1,180.50 & 15\% & 55\% \\
    \midrule
    \multirow{4}[4]{*}{Seq} & $\mathcal{I}$     &                 802.05  &  -    &  -    &  - &   8,561.53  &  -    &  -    &  - &       2,578.64  &  - & -  \\
          & $\mathcal{L}$     &                 613.26  & 24\%  & 57\%  & 27\%  &    7,556.27  &  12\% & 44\%  & 15\%  &        2,169.53 & 16\%  & 49\% \\
\cmidrule{2-13}          & $2\mathcal{I}$    &                 803.13  &  -    &  -   & -  &    6,063.16  &  -    &  -    &   - &     2,087.57  &  - & -  \\
          & $\mathcal{A}$     &                 671.66  & 16\%  & 49\%  & 19\%  &    4,741.53  &  22\% &  60\%  & 21\%  &        1,672.72 & 20\% & 58\%    
		\\ \bottomrule
	\end{tabular}
\begin{tablenotes}
		{\small 
	\item[1] For the non-sequential procedure, $n_T$ in $\texttt{GAP}^{A}_{s, n_T}(\xh)$ and $\sqrt{\smash[b]{\texttt{SV}^{A}_{s, n_T}(\xh)}}$ is set to  $n_T=500$}
	\end{tablenotes}
\end{threeparttable}
\end{table}

For all eight test problems, Table \ref{tb: all_result} provides a summary based on the experiments conducted in this section. 
It lists the number of test problems that has a higher percentage reduction in average values, variance, and bias (or relative bias) among LHS and AV.  As in previous results, we compare LHS with IID and AV with $2\mathcal{I}$ and pick the winner according to their relative performances.  
For example, looking at Table \ref{tb: BAA99_result}, we record one win for LHS in the category ``\% reduction of relative bias of $\sqrt{\texttt{SV}}$ in the non-sequential setting'' because $27\% > 18\%$.

\begin{table}[htb]
	\centering
		\caption{Summary of Results for All Test Problems with Non-Sequential (NonS) and Sequential (Seq) Approaches}
	\label{tb: all_result}
		\resizebox{\textwidth}{!}{
    \begin{threeparttable}
    \begin{tabular}{c|c|ccc|ccc|cc}
		\toprule
		\multirow{3}{*}{}               & \multirow{3}{*}{} & \multicolumn{3}{c|}{$\texttt{GAP}^{A}_{s, n_T}(\xh)$}    & \multicolumn{3}{c|}{$\sqrt{\texttt{SV}^{A}_{s, n_T}(\xh)}$} & \multicolumn{2}{c}{\texttt{CI}}                       \\
		\cline{3-10}
		  &           & \multicolumn{3}{c|}{\% reduction of} &    \multicolumn{3}{c|}{\% reduction of} &  \multicolumn{2}{c}{\% reduction of}  
		  \\  
		  \cline{3-10}
	  	   &           &    average       & variance  & bias      &  average            & variance   &  rel.bias       &      average      & variance        \\
		\hline
		    \multirow{2}[2]{*}{NonS\tnote{1}\ } & \# of $\mathcal{L}$ better     &  8     & 8     & 8     & 6     & 3     & 5     & 7 & 6  \\
          & \# of $\mathcal{A}$ better     &                 0     & 0     & 0     & 2     & 5     & 3     & 1 & 2 \\
    \midrule
    \multirow{2}[2]{*}{Seq} & \# of $\mathcal{L}$ better     &               7     & 6     & 7     & 3    & 2     & 4     &  4     & 2 \\
          & \# of $\mathcal{A}$ better     &                1     & 2     & 1     & 5    & 6     & 4     & 4     & 6
		\\ \bottomrule
	\end{tabular}
	\begin{tablenotes}
		{\small 
		\item[1]  For the non-sequential procedure, $n_T$ in $\texttt{GAP}^{A}_{s, n_T}(\xh)$ and $\sqrt{\smash[b]{\texttt{SV}^{A}_{s, n_T}(\xh)}}$ is set to $n_T=n_1$ listed in Table \ref{tb: para}
		}
	\end{tablenotes}
	\end{threeparttable}
	}
\end{table}

Let us first look at the individual \texttt{GAP} and $\sqrt{\texttt{SV}}$ estimators in the non-sequential setting.
LHS performs superior for the \texttt{GAP} estimator in this setting, yielding higher percentage reductions in all three metrics---average values, variance, and bias---for all problems tested.
We also find that LHS is more effective in reducing the average values and relative bias of $\sqrt{\texttt{SV}}$ in the non-sequential setting. 
Conversely, AV is more effective in reducing the variance of $\sqrt{\texttt{SV}}$.

For the candidate solutions $x \in X$ picked in this experiment, we estimate $\sigma_{\mathcal{A}}(x) \approx \sigma_{2\mathcal{I}}(x)$ for five problems and find that $\sigma_{\mathcal{A}}$ reduces $\sigma_{2\mathcal{I}}$ to a considerable degree for three. 
In the latter case, we may expect AV to be more effective for the $\sqrt{\texttt{SV}}$ estimator.  For BAA99-20, this appears to be true, as indicated by our estimates $\sigma_{\mathcal{A}}(\xh)=1,269.23< \sigma_{2\mathcal{I}}(\xh)=1,363.85$ whereas $\sigma_{\mathcal{L}}(\xh)=\sigma_{\mathcal{I}}(\xh)=1,928.77$.  
However, $\sigma^2_{s}(\xh)$ involves an optimal solution $x^*$; see \eqref{eq: true}. In contrast, \texttt{SV} uses an {\it estimate} of an optimal solution $x_n^*$; see \eqref{eq:svsrp} for an example \texttt{SV} that uses IID with SRP.
In addition, even if we replace $x_n^*$ by an appropriate optimal solution $x^*$, \texttt{SV} still has error due to sampling. For BAA99-20 and in general for this computational experiment, it appears that LHS improves estimates of optimal solutions and thus reduces the bias of $\sqrt{\texttt{SV}}$ (relative to IID) more than AV (relative to $2\mathcal{I}$) in the non-sequential setting.

Comparison of results on \texttt{GAP} and $\sqrt{\texttt{SV}}$ estimators between the sequential and non-sequential procedures indicate that LHS loses some effectiveness and AV gains some advantages going from non-sequential to sequential setting. 
Let us investigate this more. 
First, we numerically observe that \texttt{GAP} and $\sqrt{\texttt{SV}}$ estimators are positively correlated in both settings.  This is similar to sample means and sample variances being positively correlated in the classical settings \citep[e.g.,][]{glynn1992asymptotic}. We also find that correlation between the two estimators when the sequential procedures stop is typically higher than the non-sequential setting. 
Furthermore, LHS estimators tend to have less correlation.

Next, recall the stopping rule
\[
\texttt{GAP} \leq h'\sqrt{\texttt{SV}}+\varepsilon'. 
\] 
As discussed above, in the non-sequential setting, \texttt{GAP} of AV is relatively larger. This is because it has a larger bias. Also,  $\sqrt{\texttt{SV}}$ of AV frequently takes the smallest value among all the sampling methods.   
This is because (i) it is estimating a typically smaller variance than that of LHS and sometimes of $2\mathcal{I}$ (recall \eqref{ineq: cauchy1} and Table \ref{tb: trvar}) and (ii) it tends to reduce bias compared to $2\mathcal{I}$. 

In the sequential setting, due to the above stopping rule, having a smaller $\sqrt{\texttt{SV}}$  value forces the sequential procedure to find a smaller value of \texttt{GAP}. 
If the \texttt{GAP} is also relatively small---like for LHS---then this does not cause that much of a difference. 
In contrast, we observe that especially when the \texttt{GAP} estimator of AV takes relatively large values, then the AV sequential procedure tends to take a longer time to stop, searching for a small \texttt{GAP} that satisfies a more restrictive condition. 
This is true even though our selection of $h'$ tries to balance these effects between different sampling methods.

Moreover, due to the positive correlation between $\texttt{GAP}$ and $\sqrt{\texttt{SV}}$, when the AV sequential procedure stops, a smaller $\texttt{GAP}$ tends to induce a smaller $\sqrt{\texttt{SV}}$ relative its non-sequential counterpart. As a result, we conjecture the performance of AV on both $\texttt{GAP}$ and  $\sqrt{\texttt{SV}}$ estimators in the sequential setting improves. 
AV's performance gains in the sequential setting is not enough to match the superior performance of LHS for the \texttt{GAP} estimator. 
On the other hand, AV does gain advantage for the $\sqrt{\texttt{SV}}$ estimator in this setting.

Let us now investigate the confidence intervals. 
The non-sequential procedure has larger reductions in confidence interval widths using LHS. 
To see why this is happening, recall the following:
The  non-sequential \texttt{CI} consists of (i) \texttt{GAP} plus (ii) its standard error (\texttt{SE})  (see Table \ref{tb: nn_ss_CI}).
For \texttt{GAP}, LHS is significantly better.
For \texttt{SE}, observe that \texttt{SE} of AV and $2\mathcal{I}$ have denominator $n/2$ whereas \texttt{SE} of LHS {and IID have} denominator $n$. 
Therefore, the {\it relative} performances on \texttt{SE} are determined by the relative performances on $\sqrt{\texttt{SV}}$. Because LHS is also more effective in reducing the average $\sqrt{\texttt{SV}}$ for these experiments, LHS has larger reductions in non-sequential \texttt{CI} widths.

In contrast, Table \ref{tb: all_result} reveals that the performance of LHS and AV equalize in the sequential setting for reducing \texttt{CI} widths, with AV inducing more variance reduction. 
This is because \texttt{CI} of the sequential procedure relies mainly on the $\sqrt{\texttt{SV}}$ estimator with constant $h$ and $\varepsilon$.
Recall that $\varepsilon$ is insignificant (e.g., $2 \times 10^{-7}$).
Because $\sqrt{\texttt{SV}}$ forms the main part of this \texttt{CI} and AV tends to produce smaller  $\sqrt{\texttt{SV}}$ values in the sequential setting, AV may have an advantage in this respect.

The results in Table \ref{tb: cp} indicate that AV sequential procedures can be even more effective for these eight  problems than Table \ref{tb: all_result} suggests. 
Recall that in Table \ref{tb: cp}, AV often provided larger \texttt{CI} ratios than LHS for many test problems. 
The difference between the two tables is that the experiments summarized in Table \ref{tb: cp} have many distinct candidate solutions $\xh_T$ when the sequential procedures stop, whereas the controlled experiments of Table \ref{tb: all_result} have $\xh$ fixed.

To further investigate this, we estimated $\sigma^2_s(x_T)$ for $s=\mathcal{I, A, \text{2}I, L}$ and various $x_T$ obtained as a result of the experiments conducted in Section \ref{sssec: Comparison}. We found that for the majority of $x_T$ (but not for all), $\sigma^{2}_{\mathcal{A}}(x_T)<\sigma^{2}_{2\mathcal{I}}(x_T)$.
In fact, AV increased the variance relative to $2\mathcal{I}$ for some candidate solutions, but these were fewer in number for most problems. 
Our numerical results indicate that, even though LHS might provide more reliable estimators in small sample sizes, having $\sigma^{2}_{\mathcal{A}}(x_T)<\sigma^{2}_{2\mathcal{I}}(x_T)$ can result in a similar behavior in their sample estimators \texttt{SV} for some problems and candidate solutions.  This effect, combined with AV's performance gains in the sequential setting, helps explain AV's relatively tighter \texttt{CI} widths observed in Table \ref{tb: cp}. 
\vspace*{-0.09in}

\subsection{Recommendations}
\vspace*{-0.05in}

First, we recommend the use of A2RP over SRP.
As discussed above, SRP can falsely identify an optimal solution. 
This lowers the coverage probabilities to less-than-desired values for some problems. 
A2RP largely remedies this issue, and it has comparable or sometimes faster solution times than SRP. 

Within A2RP, if obtaining relatively tight confidence intervals is most important, then, we recommend using AV.  
However, this may come at the expense of longer solution times for some problems. 
If, on the other hand, high coverage probabilities with relatively small computation times are preferred, then LHS could be a better choice. 
LHS also can have fewer false identifications of optimal solutions and can have lower correlations between the point estimator \texttt{GAP} and its sample variance estimator \texttt{SV}, leading to more reliable results in a variety of settings. 
\vspace*{-0.05in}

\section{Conclusions and Future Work}
\label{sec: con}
\vspace*{-0.05in}

In this paper, we investigated the use of variance reduction techniques AV and LHS in sequential sampling procedures both theoretically and empirically. 
We also contrasted their use in the sequential and non-sequential settings.  
We highlight some insights gained from these results and computational experiments. 
\begin{itemize}
\item Under conditions like additivity and monotonicity of the objective function $f(x,\tilde{\xi})$ in each dimension of $\tilde{\xi}$ for each $x \in X$, LHS can be used in the sequential setting with sublinear sample sizes.  When such conditions cannot be verified, larger sample sizes must be used to guarantee asymptotic validity.  Similar conclusions hold for AV, but under different conditions. 

\item LHS is very effective for the optimality gap point estimator \texttt{GAP}, reducing both its bias and variance in all settings.  AV, on the other hand, is less effective on both accounts for the \texttt{GAP} estimator.  
When we switch to the sample variance estimator \texttt{SV}, we observe that, while LHS can provide reliable estimators in the finite-sample non-sequential setting, AV gains some advantage in the sequential setting.

\item  In the non-sequential setting, LHS provides higher reductions in confidence interval widths. This is because of the superior relative performance of LHS for both the \texttt{GAP} and sampling error portions of the non-sequential \texttt{CI}.

\item  In contrast, in the sequential setting, AV typically results in relatively smaller confidence interval widths---unlike the non-sequential setting.  This is because the sequential confidence interval is largely based on the sample variance estimator, and AV performs relatively better for this estimator in the sequential setting.

\item  The sequential stopping rule appears to induce some changes in the estimators, especially for AV. For this sampling method, the smaller absolute value of \texttt{SV} together with relatively more biased and hence larger values of \texttt{GAP} can result in  AV sequential procedures
to stop later for some problems, in search of a small \texttt{GAP}.  
This effect, combined with the positive correlation between AV's \texttt{GAP} and \texttt{SV} estimators, appears to also lower the average \texttt{SV} values for AV sequential procedures. 
This, in turn, 
improves the relative performance of AV sequential confidence intervals.

\item  Furthermore, for some problems and candidate solutions, having $\sigma^2_{\mathcal{A}}(x)<\sigma^2_{2\mathcal{I}}(x)$ appears to induce a similar behavior for the \texttt{SV} estimator in finite sample sizes. 
Consequently, for problems that have many solutions where AV reduces the variance relative to $2\mathcal{I}$, this effect gives AV an additional  advantage---especially in the sequential setting. 
Because the non-sequential \texttt{CI} is composed of \texttt{GAP} plus the sampling error and LHS performs superior for \texttt{GAP}, this relative advantage of AV  might be lessened in the non-sequential setting.

\item Overall, the results of the paper show that both LHS and AV, when used for sequential sampling in stochastic programming, provide attractive alternatives : They possess the same desired theoretical properties but provide smaller confidence interval widths with comparable solution times.  
Of the two alternative sampling schemes, when relatively tighter confidence intervals are paramount, AV can be preferred. 
When earlier stopping times with reliable estimates are most important, LHS can be chosen. 
\end{itemize}

Future work includes investigating sequential sampling procedures for  other classes of problems and  other variance reduction techniques than those considered in this paper. 
For instance, while some risk-averse problems can be put into the expectation framework presented in this paper (e.g., minimizing CVaR), sampling schemes that are aimed to find regular expectations might fail. 
Importance sampling can be useful in these cases. 
While there is some work in this area \citep{Barrera2016,kozmik_morton_15}, a thorough analysis of variance reduction schemes for sequential sampling of risk-averse and other classes of problems requires further investigation. 

Another area that merits future research is devising LHS that is aimed to `adapt' its partitioning to a particular problem.  
For instance, \citet{Xi2004Smart} consider drawing more samples in some important regions by considering the correlation structure, \citet{harris_etal_95} aim to minimize correlations among the LHS cells, and \citet{Zolan2017Optimizing} study finding partitions that minimizes the sampling variance. 
\citet{etore_etal_11} and \cite{pierre2011combined} consider adaptive stratification by adjusting the hyperrectangles and the number of samples allocated to each hyperrectangle during sequential sampling.
It would be interesting to investigate such variants of LHS and other stratified sampling techniques for risk-averse problems. 
\vspace*{-0.05in}


\baselineskip0.17in
\bibliographystyle{plainnat}
\setlength\bibsep{0.07in}
\bibliography{ref}\vspace*{-0.05in} 


\end{document}